\RequirePackage{snapshot}
\documentclass[
]{siamart}

\usepackage{microtype}
\usepackage[text={5.125in,8.25in},centering]{geometry}
\usepackage{booktabs}
\usepackage{float}
\usepackage[numbers,sort]{natbib}
\usepackage{graphicx}%
\usepackage{subfigure}

\usepackage{algorithm2e}

\SetCommentSty{mycommfont}
\SetKwComment{Comment}{[\,}{\,]}

\usepackage{braket}
\usepackage[nice]{nicefrac}

\newtheorem{prop}[theorem]{Proposition}

\theoremstyle{plain}
\theorembodyfont{\upshape}
\theoremsymbol{\ensuremath{_\Box} }
\newtheorem{example}[theorem]{Example}
\numberwithin{equation}{section}
\numberwithin{theorem}{section}
\numberwithin{table}{section}
\numberwithin{figure}{section}

\input{macros}

\def\prox#1#2{\mathbf{prox}^{\scriptscriptstyle #1}_{#2}}
\def\env#1#2{\mathbf{env}^{\scriptscriptstyle #1}_{#2}}
\newcommand{\Qsip}{\texttt{QSip}}
\newcommand{\SWinv}{\texttt{SWinv}}

\newcommand{\Qcone}{\mathbb{Q}}

\newcommand{\TheTitle}{Efficient evaluation of scaled proximal operators}
\newcommand{\TheShortTitle}{Scaled proximal operators}
\newcommand{\TheAuthors}{Michael P. Friedlander and Gabriel Goh}

\date{March 5, 2016}

\newcommand{\TitleAndAuthorCommands}{%
  \title{\MakeUppercase{\TheTitle}\thanks{\today. Research supported
      by ONR award N00014-17-1-2009.}}  \author{Michael P.
    Friedlander%
    \thanks{Departments of Computer Science and Mathematics, University of British Columbia, Vancouver, BC, V6T 1Z4, Canada 
      (\email{mpf@cs.ubc.ca}).}
    \and Gabriel Goh%
    \thanks{Department of Mathematics, University of California, Davis
      (\email{gabgohjk@gmail.com}).}}  }

\hypersetup{%
  pdftitle={\siampretitle\TheTitle},
  pdfauthor={\TheAuthors}
}

\begin{document}

\TitleAndAuthorCommands

\maketitle

\begin{abstract}
  Quadratic-support functions [Aravkin, Burke, and Pillonetto; J.\@
  Mach.\@ Learn.\@ Res. 14(1), 2013] constitute a parametric family of
  convex functions that includes a range of useful regularization
  terms found in applications of convex optimization. We show how an
  interior method can be used to efficiently compute the proximal
  operator of a quadratic-support function under different
  metrics. When the metric and the function have the right structure,
  the proximal map can be computed with cost nearly linear in the
  input size. We describe how to use this approach to implement
  quasi-Newton methods for a rich class of nonsmooth problems that
  arise, for example, in sparse optimization, image denoising, and
  sparse logistic regression.
\end{abstract}



\section{Introduction} \label{introduction}

The proximal operator is a key ingredient in many convex optimization
algorithms for problems with nonsmooth objective
functions. Proximal-gradient algorithms in particular are prized for
their applicability to a broad range of convex problems that arise in
machine learning and signal processing, and for their good theoretical
convergence properties. Under reasonable hypotheses they are
guaranteed to achieve, within $k$ iterations, a function value that is
within $\BigOh(1/k)$ of the optimal value using a constant step size,
and within $\BigOh(1/k^2)$ using an accelerated
variant. \citet{Tseng:2010} outlines a unified view of the many
proximal-gradient variations, including accelerated versions such as
FISTA~\cite{beck2009fast}.

In its canonical form, the proximal-gradient algorithm applies to
convex optimization problems of the form
\begin{equation}
  \label{eq:6}
  \minimize{x\in\Real^n} \quad f(x)+g(x),
\end{equation}
where the functions $f:\Real^n\to\Real$ and
$g:\Real^n\to\Real\cup\{+\infty\}$ are convex. We assume that $f$ has
a Lipschitz-continuous gradient, and that $g$ is lower semicontinuous
\citep[Definition~1.5]{RTRW:1998}. Typically, $f$ is a loss function
that penalizes incorrect predictions of a model, and $g$ is a
regularizer that encourages desirable structure in the
solution. Important examples of nonsmooth regularizers are the 1-norm
and total variation, which encourage sparsity in either $x$ or its
gradient. %

Suppose that $H$ is a positive-definite matrix. The iteration
\begin{equation}
  \label{eq:prox-map}
  x^{+}=\prox{H}{g}(x-H^{-1}\nabla f(x))
\end{equation}
underlies the prototypical proximal-gradient method, where $x$ is most
recent estimate of the solution, and
\begin{equation}
  \label{eq:2}
  \prox{H}{g}(z):=\argmin_{x\in\Real^n} \left\{
  \tfrac{1}{2}\|z-x\|\H^{2}+g(x) \right\}
\end{equation}
is the (scaled) proximal operator of $g$. The scaled diagonal
$H=\alpha I$ is typically used to define the proximal operator because
it leads to an inexpensive proximal iteration, particularly in the
case where $g$ is separable. (In that case, $\alpha$ has a natural
interpretation as a steplength.) More general matrices $H$, however,
may lead to proximal operators that dominate the computation. The
choice of $H$ remains very much an art. In fact, there is an entire
continuum of algorithms that vary based on the choice of $H$, as
illustrated by Table~\ref{tab:methods}.
\begin{table}
  \caption{Preconditioned proximal-gradient methods.}
  \centering
  \begin{tabular}{lll}
    \toprule
    Method & Reference & $H$
    \\\midrule
      iterative soft thresholding & \cite{beck2009fast} &$\alpha I$
    \\symmetric rank 1 & \cite{NIPS2012_4523} & $\alpha I+ss^{T}$
    \\identity-minus-rank-1 & \cite{1401.4220}  &$\alpha I-ss\T$
    \\proximal L-BFGS &\cite{schmidt2009optimizing,NIPS2014_5384,Scheinberg2016} &$\Lambda+SDS^{T}$
    \\proximal Newton &\cite{lee2014proximal,JMLR:v16:trandihn15a,Byrd2016} & $\nabla^{2}f$
    \\\bottomrule
  \end{tabular}
  \label{tab:methods}
\end{table}
The table lists a set of algorithms roughly in order of the accuracy
with which $H$ approximates the Hessian---when it exists---of the
smooth function $f$. At one extreme is iterative soft thresholding
(IST), which approximates the Hessian using a constant diagonal. At
the other extreme is proximal Newton, which uses the true Hessian. The
quality of the approximation induces a tradeoff between the number of
expected proximal iterations and the computational cost of evaluating
the proximal operator at each iteration. Thus $H$ can be considered a
preconditioner for the proximal iteration. The proposals offered by
\citet{JMLR:v16:trandihn15a} and \citet{Byrd2016} are flexible in the
choice of $H$, and so those references might also be considered to
apply to other methods listed in Table~\ref{tab:methods}.

The main contribution of this paper is to show how for an important
family of functions $g$ and $H$, the proximal operator~\eqref{eq:2}
can be computed efficiently via an interior method. This approach
builds on the work of \citet{aravkin2013sparse}, who define the class
of quadratic-support functions and outline a particular interior
algorithm for their optimization.  Our approach is specialized to the
case where the quadratic-support function appears inside of a proximal
computation.  Together with the correct dualization approach
(\S\ref{sec:prox-to-qp}), this yields a particularly efficient
interior implementation when the data that define $g$ and $H$ have
special structure (\S\ref{sec:eval-prox-oper}). The proximal
quasi-Newton method serves as a showcase for how this technique can be
used within a broader algorithmic context (\S\ref{sec:quasi-Newton}).

\section{Quadratic-support functions} \label{sec:problem-setup}

\citet{aravkin2013sparse} introduce the notion of a
quadratic-support (QS) function, which is a generalization of
sublinear support functions~\cite[Ch.~8E]{RTRW:1998}. Here we
introduce a slightly more general definition than the version
implemented by Aravkin et al.\@ We retain the ``QS'' designation
because the quadratic term, which is an essential feature of their
definition, can also be expressed by the version we use here.

Let $\Ball_p=\set{z|\norm{z}_p\le1}$ and
$\Kscr = \Kscr_1 \times \cdots \times \Kscr_k$, where each cone
$\Kscr_i$ is either a nonnegative orthant $\Real_+^m$ or a
second-order cone
$\Qcone^m=\set{(\tau,z)\in\Real\times\Real^{m-1} | z\in\tau\Ball_2}$.
(The size $m$ of the cones may of course be different for each index
$i$.)  The notation $Ay\succeq_\Kscr b$ means that $Ay-b\in\Kscr$, and
$\tau\Ball_p\equiv\set{\tau z | z \in\Ball_p}$. The indicator on a
convex set $U$ is denoted as
\[
\delta(x\mid U) = \begin{cases}
                           0        & \mbox{if $x\in U$,}
                         \\ +\infty & \mbox{otherwise.}
                          \end{cases}
\]
Unless otherwise specified, $x$ is an $n$-vector.  To help
disambiguate dimensions, the $p$-by-$p$ identity matrix is denoted by
$I_p$, and the $p$-vector of all ones by $\one_p$.

We consider the class of functions
$g:\Real^n\to\Real\cup\{+\infty\}$ that have the conjugate
representation
\begin{equation}
  \label{eq:qs}
  g(x) = \sup_{y\in\Yscr}\  y\T(Bx+d),
  \text{where}
  \Yscr = \{y\in\Real^\ell\mid Ay \succeq_\Kscr b\}.
\end{equation}
(The term ``conjugate'' alludes to the implicit duality, since $g$ may
be considered as the conjugate of the indicator function to the set
$\Yscr$.)  We assume throughout that the feasible set
$\Yscr=\set{y | Ay\succeq_\Kscr b}$ is nonempty. If $\Yscr$ contains
the origin, the QS function $g$ is nonnegative for all $x$, and we can
then consider it to be a penalty function. This is automatically true,
for example, if $b\le0$.

The formulation~\eqref{eq:qs} is close to the standard definition of a
sublinear support function \cite[\S13]{Roc70}, which is recovered by
setting $d=0$ and $B=I$, and letting $\Yscr$ be any convex set. Unlike
a standard support function, $g$ is not positively homogeneous if
$d\ne0$. This is a feature that allows us to capture important classes
of penalty functions that are not positively homogeneous, such as
piecewise quadratic functions, the ``deadzone'' penalty, or indicators
on certain constraint sets. These are examples that are not
representable by the standard definition of a support function.  Our
definition springs from the quadratic-support function definition
introduced by \citet{aravkin2013sparse}, who additionally allow for an
explicit quadratic term in the objective and for $\Yscr$ to be any
nonempty convex set. The concrete implementation considered by
\citeauthor{aravkin2013sparse}, however, is restricted to the case
where $\Yscr$ is polyhedral. In contrast, we also allow $\Kscr$ to
contain second-order cones. Therefore, any quadratic objective terms
in the \citeauthor{aravkin2013sparse}\ definition can be ``lifted''
and turned into a linear term with a second-order-cone constraint (see
Example~\ref{ex:2-norm}). Our definition is thus no less general.

This expressive class of functions includes many penalty functions
commonly used in machine learning; \citeauthor{aravkin2013sparse}\
give many other examples. In addition, they show how to interpret QS
functions as the negative log of an associated probability density,
which makes these functions relevant to maximum a posteriori
estimation.  In the remainder of this section we provide some examples
that illustrate various regularizing functions and constraints that
can be expressed as QS functions.

\begin{example}[1-norm regularizer] \label{L1_Example}

  The 1-norm has the QS representation
  \[
    \norm{x}_1 = \sup_y\set{y\T x | y\in\Ball_\infty},
  \]
  where
\begin{equation} \label{eq:l1-rep}
      A = \pmat{\phantom-I_n\\-I_n},
\quad b = -{\pmat{\one_n\\\one_n}},
\quad d= 0, \quad B=I_n,
\quad \Kscr = \Real_+^{2n},
\end{equation}
\end{example}

\begin{example}[2-norm] \label{ex:2-norm} This simple example
  illustrates how the QS representation~\eqref{eq:qs} can represent
  the 2-norm, which is not possible using the QS formulation described
  by \citet{aravkin2013sparse} where the constraints are
  polyhedral. With our definition, the 2-norm has the QS
  representation
\[
\norm{x}_2 =
\sup_y\set{y\T x | y\in\Ball_2} =
\sup_y\set{ y\T x | (1,y)\succeq_{\Kscr}0},
\]
where
\begin{equation}\label{eq:1}
        A = \pmat{0\\ I_n},
  \quad b = \pmat{1\\ 0},
  \quad d = 0,
  \quad B = I_n,
  \quad \Kscr = \Qcone^{n+1}.
\end{equation}
\end{example}

\begin{example}[Polyhedral norms] \label{ex:norms} Any
  polyhedral seminorm is a support function, e.g., $\norm{Bx}_1$ for
  some matrix $B$. In particular, if the set $\set{y|Ay\ge b}$
  contains the origin and is centro-symmetric, then
  \[
    \norm{x} := \sup_y\set{ y\T Bx | Ay\ge b}
  \]
  defines a norm if $B$ is nonsingular, and a seminorm otherwise.
  This is a QS function with $d:=0$ (as will be the case for any
  positively homogeneous QS function) and $\Yscr:=\set{y|Ay\ge b}$.
\end{example}

\begin{example}[Quadratic function]\label{ex:quadratic-function}
  This example justifies the term ``quadratic'' in our modified
  definition, even though there are no explicit quadratic terms. It
  also illustrates the roles of the terms $B$ and $d$.  The quadratic
  function can be written as
  \[
  \half\norm{x}_2^2
  = \sup_{y,\,t}\Set{\! y\T x - \half t | \norm{y}^2_2\le t\! }
  = \sup_{y,\,t}
    \Set{\!
      \pmat{y\\t}\T\left[\pmat{I_n\\0}x-\pmat{0\\\nicefrac{1}{2}}\right] |
      \norm{y}_2^2\le t
    \!}.
  \]
  Use the derivation in \ref{sec:qs-soc} to obtain the QS
  representation with parameters
  \[
          A = \pmat{0&\nicefrac{1}{2}\\0&\nicefrac12\\I_n&0},
    \quad b = \pmat{{\phantom-\nicefrac12}\\-\nicefrac12\\\phantom-0},
    \quad d = \pmat{0\\\nicefrac12},
    \quad B = \pmat{I_n\\0},
    \quad \Kscr = \Qcone^{n+2}.
\]
\end{example}

\begin{example}[1-norm constraint]
  \label{ex:1-norm-ball}
  This example is closely related to the 1-norm regularizer in
  Example~\ref{L1_Example}, except that the QS function is used to
  express the constraint $\norm{x}_1\le 1$ via an indicator function
  $g=\delta(\,\cdot\mid\Ball_1)$.  Write the indicator to the 1-norm
  ball as the conjugate of the infinity norm, which gives
  \begin{equation}
    \label{eq:10}
\begin{aligned}
  \delta(x\mid \Ball_1)
     &= \sup_{y}\set{y\T x - \|y\|_\infty}
  \\ &= \sup_{y,\,\tau}\set{y\T x-\tau | y\in\tau\Ball_\infty }
      = \sup_{y,\,\tau}\set{y\T x-\tau | -\tau\one_n \leq y\leq \tau\one_n }.
\end{aligned}
\end{equation}
This is a QS function with parameters
\[
        A = \pmat{-I_n&\one_n\\ \phantom-I_n & \one_n},
\quad b = 0,
\quad d = \pmat{\phantom-0\\-1},
\quad B = \pmat{I_n\\0},
\quad \Kscr = \Real^{2n}.
\]
\end{example}

\begin{example}[Indicators on polyhedral cones]
Consider the following polyhedral cone and its polar:
\[
  U=\set{x|Bx\le0} \text{and} U^\circ=\set{B\T y|y\le 0}.
\]
Use the support-function representation of a cone in terms of its
polar to obtain
\begin{equation}\label{eq:12}
\delta(x\mid U) = \delta^*(\cdot\mid U^\circ)(x)
                = \sup_y \set{y\T B x | y\leq 0 },
\end{equation}
which is an example of an elementary QS function. (See
\citet{RTRW:1998} for definitions of the polar of a convex set, and
the convex conjugate.) A concrete example is the positive orthant,
obtained by choosing $B = I_n$. An important example, used in isotonic
regression~\cite{Best1990}, is the monotonic cone
$$
  U:=\set{x|x_i\ge x_j,\ \forall (i,j)\in\Escr},
$$
Here, $\Escr$ is the set of edges in a graph $\Gscr=(\Vscr,\Escr)$
that describes the relationships between variables in $\Vscr$. If we
set $B$ to be the incidence matrix for the graph, \eqref{eq:12} then
corresponds to the indicator on the monotonic cone $U$.
\end{example}

\begin{example}[Distance to a cone] The distance to a cone $U$ that is
  a combination of polyhedral and second-order cones can be
  represented as a QS function:
  \begin{align*}
  \inf_{x\in U}\|x-y\|_2
  &= \inf_{x}\ \set{\norm{x-y}_2+\delta(x\mid U)}
 \\ &= \big[ \delta(\cdot\mid \Ball_2)+\delta(\cdot\mid U^{\circ})\big]^{*}(y)
  =\ \sup \left\{ y\T x \mid y\in \Ball_2\cap U^{\circ}\right\}.
\end{align*}
The second equality follows from the relationship between infimal
convolution and conjugates \cite[\S16.4]{Roc70}.  When $U$ is the
positive orthant, for example, $g(x) = \|\max\{0,\,x\}\|_2,$ where the
\emph{max} operator is taken elementwise.
\end{example}

\section{Building quadratic-support functions}
\label{sec:qs-calculus}

Quadratic-support functions are closed under addition, composition
with an affine map, and infimal convolution with a quadratic
function. In the following, let $g_i$ be QS functions with parameters
$A_{i}$, $b_{i}$, $d_{i}$, $B_{i}$, and $\Kscr_i$ (with
$i=0,1,2$). The rules for addition and composition are described in
\cite{aravkin2013sparse}, which are here summarized and amplified.

\paragraph{Addition rule}
The function \[h(x):=g_1(x) + g_2(x)\] is QS with parameters
\begin{equation*}
       A = \pmat{A_1 & \\ & A_2},
\quad  b = \pmat{b_{1}\\b_{2}},
\quad  d = \pmat{d_{1}\\d_{2}},
\quad  B = \pmat{B_{1}\\B_{2}},
\quad  \Kscr = \Kscr_1\times \Kscr_2.
\end{equation*}

\paragraph{Concatenation rule} The function
\[h(x) := g_0(x^1) + \cdots + g_0(x^k),\]
where each partition $x^i\in\Real^n$, is QS with parameters
\vspace{-1.5mm}
\begin{equation}
  \label{Calculus-Rule-Concat}
        (A,\,B) = I_k \otimes (A_0,\,B_0),
\qquad  (b,\,d) = \one_k \otimes (b_0,\,d_0),
\qquad  \Kscr = \Kscr_0 \times \overset{(k)}{\cdots} \times \Kscr_0.
\end{equation}
where the symbol $\otimes$ denotes the Kronecker product. The rule for
concatenation follows from the rule for addition of QS functions.

\paragraph{Affine composition rule}
The function \[h(x) := g_0(Px-p)\] is QS with parameters
\begin{equation*}
       A = A_0,
\quad  b = b_0,
\quad  d = d_0 - B_0p,
\quad  B = B_0P.
\end{equation*}

\paragraph{Moreau-Yosida regularization} The Moreau-Yosida envelope of
$g_0$ is the value of the proximal operator, i.e.,
\begin{equation}\label{eq:4}
\env{H}{g_0}(z):=\inf_{x}\set{\tfrac{1}{2}\|z-x\|\H^{2}+g_0(x)}.
\end{equation}
It follows from \citet[Proposition~4.10]{BurkeHoheisel:2013} that
\[
 \env{H}{g_0}(z)
 = \sup_y\set{y\T(B_0x+d_0)-\half y\T B_0H\inv B_0\T y | A_0y\succeq_{\Kscr_0} b_0},
\]
which is a QS function with parameters
\begin{equation*}
       A = \pmat{0&\nicefrac{1}{2}\\0&\nicefrac12\\R&0\\A_0&0},
\quad  b = \pmat{\phantom-\nicefrac12\\-\nicefrac12\\0\\b_0},
\quad  d = \pmat{d_0\\-\nicefrac12},
\quad  B = \pmat{B_0\\0},
\quad  \Kscr = \Qcone^{n+2}\times \Kscr_0,
\end{equation*}
with Cholesky factorization $R\T R = B_0H\inv B_0^T$. The derivation
is given in \ref{sec:qs-soc}, where we take $Q=B_0H\inv B_0^T$.

\begin{example}[Sums of norms] \label{ex:sum-of-norms} In applications
  of group sparsity \citep{YuL06,jenatton2010proximal}, various norms
  are applied to all partitions of $x=(x^1,\ldots,x^p)$, which
  possibly overlap. This produces the QS function
  \begin{equation} \label{eq:sum-of-norms-g}
    g(x) = \|x^1\| + \cdots + \|x^p\|,
  \end{equation}
  where each norm in the sum may be different. In particular, consider
  the case of adding two norms
  $g(x)=\norm{x^1}_\triangledown + \norm{x^2}_\vartriangle$. (The
  extension to adding three or more norms follows trivially.) First,
  we introduce matrices $P_i$ that restrict $x$ to partition $i$,
  i.e., $x^i=P_ix$, for $i=1,2$. Then
  \[
    g(x) = \norm{P_1 x}_\triangledown + \norm{P_2 x}_\vartriangle.
  \]
  Then we apply the affine-composition and addition rules to
  determine the corresponding quantities that define the QS
  representation of $g$:
  \[
          A = \pmat{A_\triangledown \\& A_\vartriangle},
    \quad b = \pmat{b_\triangledown\\b_\vartriangle},
    \quad d = 0,
    \quad B = \pmat{B_\triangledown P_1\\B_\vartriangle P_2},
    \quad \Kscr = \Kscr_\triangledown \times \Kscr_\vartriangle,
  \]
  where $A_i$, $b_i$, $B_i$, and $\Kscr_i$ (with
  $i=\triangledown,\vartriangle$) are the quantities that define the
  QS representation of the individual norms. (Necessarily, $d=0$
  because the result is a norm and therefore positive homogeneous.) In
  the special case where $\norm{\cdot}_\triangledown$ and
  $\norm{\cdot}_\vartriangle$ are both the 2-norm, then $A_i$, $b_i$,
  $B_i$, and $\Kscr_i$ are given by~\eqref{eq:1} in
  Example~\ref{ex:2-norm}.
\end{example}

\begin{example}[Graph-based 1-norm and total variation]
  \label{ex:graph1-norm-example}

  A variation of Example~\ref{ex:sum-of-norms} can be used
  to define a variety of interesting norms, including the graph-based
  1-norm regularizer used in machine learning \citep{chin2013runtime},
  and the isotropic and anisotropic versions of total variation (TV),
  important in image processing \citep{LouZengOsherYin:2014}. Let
  \begin{equation*}
    g(x)=\|Nx\|_{G}
    \quad\text{with}\quad \|z\|_G = \sum_{i=1}^p \|z^i\|_2,
  \end{equation*}
  where $z^i$ is a partition of $z$ and $N$ is an $m$-by-$n$
  matrix. For anisotropic TV and the graph-based 1-norm regularizer,
  $N$ is the adjacency matrix of a graph, and each partition $z^i$ has
  a single unique element, so $g(x)=\norm{Nx}_1$.  For isotropic TV,
  each partition captures neighboring pairs of variables, and $N$ is a
  finite-difference matrix.  The QS addition and affine-composition
  rules can be combined to derive the parameters of $g$. When $p=m$
  (i.e., each $z^i$ is a scalar), we are summing $n$ absolute-value
  functions, and we use \eqref{eq:l1-rep} and
  \eqref{Calculus-Rule-Concat} to obtain
  \begin{equation} \label{eq:qs-rep-graph-1-norm}
        A = I_m \otimes \pmat{\phantom-1\\-1},
  \quad b = \one_m \otimes \pmat{-1\\-1},
  \quad d = 0,
  \quad B = N,
  \quad \Kscr = \Real^{2m}_+.
  \end{equation}
  Now consider the variation where $p=m/2$, (i.e., each partition has
  size 2), which corresponds to summing $m/2$ two-dimensional
  2-norms. Use \eqref{eq:1} to obtain
\[
        A = I_{m/2} \otimes \pmat{0\\ I_2},
  \quad b = \one_{m/2} \otimes \pmat{1\\ 0},
  \quad d = 0,
  \quad B = N,
  \quad \Kscr = \Qcone^2 \times \overset{(m/2)}{\cdots} \times \Qcone^2.
\]
\end{example}

\section{The proximal operator as a conic QP} \label{sec:prox-to-qp}

We describe in this section how the proximal map~\eqref{eq:prox-map}
can be obtained as the solution of a quadratic optimization problem
(QP) over conic constraints,
\begin{align} \label{eq:qp}
  \minimize{y}\quad\tfrac12 y\T Qy-c\T y \quad\st\quad Ay\succeq_\Kscr b,
\end{align}
for some positive semidefinite $\ell$-by-$\ell$ matrix $Q$ and a
convex cone $\Kscr=\Kscr_1\times\cdots\times\Kscr_k$.  The
transformation to a conic QP is not immediate because the
definition of the QS function implies that the proximal map involves
nested optimization. Duality, however, furnishes a means for
simplifying this problem.

\begin{prop} \label{prop:prox-to-qp} Let $g$ be a QS function. The following problems are dual pairs:
\begin{subequations} \label{eq:dual-probs}
\begin{alignat}{2}
  &\minimize{x} &\quad  &\half\norm{z-x}^2\H+g(x),
  \label{eq:primal-prob}
\\
  &\minimize{Ay\succeq_\Kscr b} &\quad &\half y\T BH^{-1}B^{T} y-(d+Bz)\T y.
  \label{eq:dual-prob}
\end{alignat}
\end{subequations}
If strong duality holds, the corresponding primal-dual solutions are
related by
\begin{equation} \label{eq:3}
  Hx+B\T y=Hz.
\end{equation}
\end{prop}
\begin{proof}
Let
\[
  h_1(x) := \tfrac{1}{2}\|x-z\|_{H}^{2}
  \text{and}
  h_2(x) : =\sup_{y\in\Yscr} \set{ y\T(x+d)}.
\]
If strong duality holds, it follows from
\citet[Prop.~5.3.8]{bertsekas2009convex} that
\begin{equation} \label{eq:strong-duality}
  \inf_{x}\ h_1(x) + h_2(Bx)
  =- \inf_{y}\ h_1^*(-B^T y) + h_2^*(y),
\end{equation}
where
\[
  h_1^{*}(y)  = \tfrac{1}{2}\|y\|_{H^{-1}}^{2}+z^{T}y
  \text{and}
  h_2^{*}(y)  = \delta(z \mid \Yscr)-d^{T}y
\]
are the Fenchel conjugates of $h_1$ and $h_2$, and the infima on
both sides are attained. (See \citet[\S12]{Roc70} for the convex
calculus of Fenchel conjugates.)  The right-hand side
of~\eqref{eq:strong-duality} is precisely the dual
problem~\eqref{eq:dual-prob}. It also follows from Fenchel duality that the pair $(x,y)$
is optimal only if
\[
  x \in \argmin_x\set{ h_1(x) + y\T B x }.
\]
Differentiate this objective to obtain~\eqref{eq:3}.
\end{proof}

Strong duality holds when
$B\cdot\mbox{ri dom}(h_{1})\cap\mbox{ri dom}(h_{2})\neq\emptyset.$
This holds, for example, when the interior of the domain of $g$ is nonempty, since
\[
  \mbox{int dom}(g)\neq\emptyset
  \iff
  \mbox{im}\,B\cap\mbox{int dom}(h_{2})\neq\emptyset
  \Rightarrow
  B\cdot\mbox{ri dom}(h_{1})\cap\mbox{ri dom}(h_{2})\neq\emptyset.
\]
In all of the examples in this paper, this holds true.
\section{Primal-dual methods for conic QP}

Proposition~\ref{prop:prox-to-qp} provides a means of evaluating the
proximal map of QS functions via conic quadratic optimization. There
are many algorithms for solving convex conic QPs, but primal-dual
methods offer a particularly efficient approach that can leverage the
special structure that defines the class of QS functions.  A detailed
discussion of the implementation of primal-dual methods for conic
optimization is given by \citet{vandenberghe:2010}; here we summarize
the main aspects that pertain to implementing these methods
efficiently in our context.

The standard development of primal-dual methods for~\eqref{eq:qp} is
based on perturbing the optimality conditions, which can be stated as
follows. The solution $y$, together with slack and dual vectors $s$
and $v$, must satisfy
\[
  Qy-A^{T}v = c,\quad v \succeq_\Kscr 0, \quad Sv = 0,
\]
where the matrix $S$ is block diagonal, and each $m_i$-by-$m_i$ block
$S_i$ is either a diagonal or arrow matrix depending on the type of
cone, i.e.,
\[
S_i = \begin{cases}
        \diag(s_i)  & \mbox{if $\Kscr_i=\Real^{m_i}_+$,}
      \\\arrow(s_i) & \mbox{if $\Kscr_{i}=\Qcone^{m_i}$,}
      \end{cases}
\qquad
\arrow(u) := \pmat{u_0 & \bar u^T \\ \bar u & u_0 I}
\text{for}
u = (u_0, \bar u).
\]
See \citet{vandenberghe:2010} for further details.

Now replace the complementarity condition $Sv=0$ with its perturbation
$Sv=\mu e $, where $\mu$ is a positive parameter and
$e=(e^1,\ldots,e^k)$, with each partition defined by
\[
e^i =
\begin{cases}
  (1,1,\ldots,1) & \mbox{if $\Kscr_{i}=\Real_+^{m_i}$,}
\\(1,0,\ldots,0) & \mbox{if $\Kscr_{i}=Q^{m_i}$.}
\end{cases}
\]
A Newton-like method is applied to the perturbed optimality
conditions, which we phrase as the root of the function
\begin{equation}\label{eq:int_residual}
R_{\mu}:\pmat{y\\v\\s}
\mapsto
\pmat{r_d\\r_p\\r_\mu} :=
\pmat{
  Qy -A\T v -c
\\Ay-s-b
\\Sv-\mu e
}.
\end{equation}
Each iteration of the method proceeds by systematically choosing the
perturbation parameter $\mu$ (ensuring it decreases), and obtaining
each search direction as the solution of the Newton system
\begin{equation} \label{eq:5}
\begin{pmatrix}
        Q & -A^{T} & 0
      \\A & 0 & -I
      \\0 & S & V
\end{pmatrix}
\pmat{
  \Delta y\\
  \Delta v\\
  \Delta s
}
=-{\pmat{r_d\\r_p\\r_\mu}},
\qquad
\begin{pmatrix}y^{+}\\v^{+}\\s^{+}\end{pmatrix}
=
\pmat{y\\v\\s}
+\alpha
\pmat{\Delta y\\\Delta v\\\Delta s}.
\end{equation}
The steplength $\alpha$ is chosen to ensure that $(v^+,s^+)$ remain
in the strict interior of the cone.

One approach to solving for the Newton direction is to apply block
Gaussian elimination to~\eqref{eq:5}, and obtain the search direction
via the following systems:
\begin{subequations}
\begin{align}
   (Q+A\T S^{-1}VA)\Delta x & =r_d+A\T S^{-1}(V r_p+r_\mu),
                             \label{eq:newton-sc}
\\ \Delta v & =S^{-1}(Vr_p+r_\mu-VA\Delta x),
\\ \Delta s & =V^{-1}\left(r_\mu-S\Delta v\right).
\end{align}
\end{subequations}
In practice, the matrices $S$ and $V$ are rescaled at each iteration
in order to yield search directions with favorable properties. In
particular, the Nesterov-Todd rescaling redefines $S$ and $V$ so that
$S V\inv = \block(u)$ for some vector $u$, where
\begin{equation} \label{eq:block-diag-def}
  \block(u)_i =
  \begin{cases}
    \diag(u_i)                      & \mbox{if $\Kscr_{i} = \Re_+^{m_i}$,}
  \\(2 u_i u_i^T- [u_i\T Ju_i] J)^2 & \mbox{if $\Kscr_{i} = \Qcone^{m_i}$,}
  \end{cases}
  \qquad
  J = \pmat{
1 & 0\\
0 & -I_{(m_i-1)}
}.
\end{equation}
The cost of the overall approach is therefore determined by the cost of
solving, at each iteration, linear systems with the matrix
\label{InteriorL}
\begin{equation}\label{eq:9}
\Lscr(u) := Q+A^T \mbox{\bf{block}}(u)^{-1} A,
\end{equation}
which now defines the system~\eqref{eq:newton-sc}.

\section{Evaluating the proximal operator} \label{sec:eval-prox-oper}

We now describe how to use Proposition~\ref{prop:prox-to-qp} to
transform a proximal operator~\eqref{eq:2} into a conic QP that can be
solved by the interior algorithm described in~\S\ref{InteriorL}. In
particular, to evaluate $\prox{H}{g}(x)$ we solve the conic
QP~\eqref{eq:qp} with the definitions
\begin{equation} \label{eq:8}
 Q := BH^{-1}B^T, \qquad c := d + Bx;
\end{equation}
the other quantities $A$, $b$, and the cone $\Kscr$, appear
verbatim. Algorithm~\ref{algo:proxAlgo} summarizes the procedure.  As
we noted in \S\ref{InteriorL}, the main cost of this procedure is
incurred in Step~1, which requires repeatedly solving linear systems
that involve the linear operator~\eqref{eq:9}. Together
with~\eqref{eq:8}, these matrices have the form
\begin{equation} \label{eq:L-def}
 \Lscr(u)=BH^{-1}B^{T}+A\T\block(u)\inv A.
\end{equation}

\begin{algorithm}[t]
  \smallskip{}
  \begin{tabular}{@{}l@{\ }l}
    \textsc{Input} &: $x$, $H$, and QS function $g$ as defined by parameters
                             $A$, $b$, $d$, $B$, $\Kscr$
    \\[4pt]
    \textsc{Output}&: $\prox{H}{g}(x)$
  \end{tabular}

  \smallskip{}\smallskip{}

  \textbf{Step 1}: Apply interior method to QP~\eqref{eq:dual-prob} to
  obtain $\ystar$.

  \smallskip{}

  \textbf{Step 2}: Return $H^{-1}(c-B\T \ystar)$.\smallskip{}

  \caption{Evaluating $\prox H g(x)$}
  \label{algo:proxAlgo}
\end{algorithm}

Below we offer a tour of several examples, ordered by level of
simplicity, to illustrate the details involved in the application of
our technique.  The Sherman-Woodbury (SW) identity
\[
 (D + UMU^T)\inv = D\inv - D\inv U(M\inv+U\T D\inv U)\inv U\T D\inv,
\]
valid when $M\inv+U\T D\inv U$ is nonsingular, proves useful for
taking advantage of certain structured matrices that arise when
solving~\eqref{eq:L-def}. Some caution is needed, however, because
it is known that the SW identity can be numerically unstable
\citep{yip:1986}.%

For our purposes, it is useful to think of the SW formula as a routine
that takes the elements ($D,U,M$) that define a linear operator
$D + UMU^T$, and returns the elements ($D_1, U_1, M_1$) that define
the inverse operator $D_1 +U_1M_1U_1^T = (D + UMU^T)^{-1}$. We assume
that $D$ and $M$ are nonsingular. The following pseudocode summarizes
the operations needed to compute the elements of the inverse operator.

\smallskip

\RestyleAlgo{plain}
\begin{algorithm}[H]
  \SetKwInOut{Input}{input}\SetKwInOut{Output}{output}
  \SetKwProg{Fn}{function}{}{end}
  \SetAlgoNoLine
  \DontPrintSemicolon
  \Fn{\tt SWinv($D, U, M$)}{
  \nl $D_1 \gets D\inv$\;
  \nl $U_1 \,\gets D_1 U$\;
  \nl $M_1 \gets (M\inv + U\T U_1)\inv$\;
  \nl\KwRet{$D_1$, $U_1$, $M_1$}
  }
\end{algorithm}

\smallskip

\noindent Typically, $D$ is a structured operator that admits a fast
algorithm for solving linear systems with any right-hand side, and $U$
and $M$ are stored explicitly as dense
matrices. %
Step~1 above computes a new operator $D_1$ that simply interchanges
the multiplication and inversion operations of $D$. Step~2 applies the
operator $D_1$ to every column of $U$ (typically a tall matrix with
few columns), and Step~3 requires inverting a small matrix.

\begin{example}[1-norm regularizer; cf.\@ Example~\ref{L1_Example}]
\label{L1_Example-cont}

Example~\ref{L1_Example} gives the QS representation for
$g(x)=\norm{x}_1$, and the required expressions for $A$, $B$, and
$\Kscr$. Because $\Kscr$ is the nonnegative orthant,
$\block(u) = \diag(u)$; cf.~\eqref{eq:block-diag-def}. With the
definitions of $A$ and $B$, the linear operator $\Lscr$
in~\eqref{eq:L-def} simplifies to
\begin{align*}
  \Lscr(u) = H^{-1}+A\T\diag(u)A
           = H^{-1}+\Sigma,
\end{align*}
where $\Sigma$ is a positive-definite diagonal matrix that depends on
$u$. If it happens that the preconditioner $H$ has a special structure
such that $H+\Sigma\inv$ is easily invertible, it may be convenient to
apply the SW identity to obtain equivalent formulas for the inverse
\[
  \Lscr(u)^{-1} = (H\inv+\Sigma)\inv = H - H (H+\Sigma\inv)\inv H.
\]
Banded, chordal, and diagonal-plus-low-rank matrices are examples of
specially structured matrices that make one of these formulas for
$\Lscr\inv$ efficient. They yield the efficiency because subtracting
the diagonal matrix $\Sigma$ preserves the structure of either $H$ or
$H\inv$.
\end{example}

In the important special case where $H=\diag(h)$ is diagonal,
each component $i$ of the proximal operator for the 1-norm can be
obtained directly via the formula
\[
 [\prox{H}{g}(x)]_i =\sign(x_i)\cdot\max\{\abs{x_i}-1/h_{ii},\, 0\},
\]
$h_{ii}$ are the diagonal components of $H$. This corresponds to the
well-known soft-thresholding operator. There is no simple formula,
however, when $H$ is more general.

\begin{example}[Graph-based 1-norm]
  \label{ex:graph-1-norm}

  Consider the graph-based 1-norm function from
  Example~\ref{ex:graph1-norm-example} induced by a graph $\Gscr$
  with adjacency matrix $N$. Substitute the definitions of $A$ and
  $B$ from \eqref{eq:qs-rep-graph-1-norm} into the formula for $\Lscr$
  and simplify to obtain
  \begin{equation*}
    \Lscr(u) = NH\inv N^T + A\T \diag(u)A = NH\inv N^T + \Sigma,
  \end{equation*}
  where $\Sigma:=A\T\diag(u)A$ is a positive-definite diagonal
  matrix. (As with Example~\ref{L1_Example-cont}, $\Kscr$ is the
  positive orthant, and thus $\block(u) =
  \diag(u)$.) %
  Linear systems of the form $\Lscr(u)p=q$ then can be solved with the
  following sequence of operations, in which we assume that
  $H=\Lambda+UMU^T$, where $\Lambda$ is diagonal.

  \begin{algorithm}
    \SetKwInOut{Input}{input}\SetKwInOut{Output}{output}
    \DontPrintSemicolon
    \nl$(\Lambda_1, U_1, M_1) \gets \SWinv(\Lambda, U, M)$
    \Comment*{$H\inv \equiv \Lambda_1 + U_1M_1U_1^T$}
    \nl$\Sigma_1 \gets N\Lambda_1N^T + \Sigma$\;
    \nl$(\Sigma_2, U_2, M_2) \gets \SWinv(\Sigma_1, NU_1, M_1)$
    \Comment*{$\Lscr(u)\inv\equiv \Sigma_2+U_2M_2U_2^T$}
    \nl$p \gets \Sigma_2 q + U_2 M_2 U_2\T q$ \Comment*{solve
      $\Lscr(u)p=q$} \nl\KwRet{$p$}
  \end{algorithm}

  \noindent Observe from the definition of $H$ and the definition of
  $\Sigma_1$ in Step 2 above that
  \[
    \Lscr(u) = \Sigma_1 + NU_1 M_1 U_1^T N^T,
  \]
  and then Step~3 computes the quantities that define the inverse of
  $\Lscr$.  The bulk of the work in the above algorithm happens in
  Step~3, where $\Sigma_2\equiv\Sigma_1\inv$ is applied to each
  column of $NU_1$ (see Step~2 of the \SWinv\ function), and in
  Step~4, where $\Sigma_2$ is applied to $q$.  Below we give two
  special cases where it is possible to take advantage of the
  structure of $N$ and $H$ in order to apply $\Sigma_2$ efficiently to
  a vector.

  \smallskip

\begin{description}
\item[1-dimensional total variation.] Suppose that the graph $\Gscr$
  is a path. Then the $(n-1)\times n$ adjacecy matrix is given by
  \[
    N =  \pmat{-1 & 1\\ & \ddots& \ddots\\&&-1 & 1}.
  \]
  The matrix $\Sigma_1:=N\Lambda\inv N^T+\Sigma$ (see Step~2 of the
  above algorithm) is tridiagonal, and hence equations of the form
  $\Sigma_1q = p$ can be solved efficiently using standard techniques,
  e.g., \citet[Algorithm 4.3.6]{GoluLoan:1989}.

\smallskip

\item[Chordal graphs.] If the graph $\Gscr$ is chordal, than the
  matrix $N\T D N$ is also chordal when $D$ is diagonal. This implies
  that it can be factored in time linear with the number of edges of
  the graph \citep{Andersen:2010}. We can use this fact to apply
  $\Sigma_2\equiv\Sigma_1\inv$ efficiently, as follows: let
  $(\Sigma_3, U_3, M_3)=\SWinv(\Sigma, N, \Lambda_1)$, which implies
  \[
          \Sigma_2 := \Sigma_3 + U_3M_3U_3^T,
    \quad\mbox{where}
    \quad \Sigma_3 := \Sigma\inv,
    \quad M_3 := (N\T\Sigma\inv N + \Lambda_1)\inv.
  \]
  Because $N\T\Sigma\inv N$ is chordal, so is $M_3$, and any methods
  efficient for solving with chordal matrices can be used when
  applying $\Sigma_2$.
\end{description}
\end{example}

\begin{example}[1-norm constraint; cf.\@ Example~\ref{ex:1-norm-ball}]
\label{ex:1-norm-ball-v2}

Example~\ref{ex:1-norm-ball} gives the QS representation for the
indicator function on the 1-norm ball. Because the constraints on $y$
in~\eqref{eq:10} involve only bound constraints,
$\block(u)=\diag(u)$. With the definitions of $A$ and $B$ from
Example~\ref{ex:1-norm-ball}, the linear operator $\Lscr$ has the form
\[
\Lscr(u)
=
\begin{pmatrix}0\\I_n\end{pmatrix}
H^{-1}
\begin{pmatrix} 0 & I_n\end{pmatrix}
+
\begin{pmatrix}\phantom-\one_n^T & \one_n^T \\ -I_n & I_n\end{pmatrix}
\begin{pmatrix}\diag(u_{1})\\ & \diag(u_{2})\end{pmatrix}
\begin{pmatrix}\one_n&-I_n\\\one_n&\phantom-I_n\end{pmatrix},
\]
where $u = (u_1, u_2)$. Thus, $\Lscr$ simplifies to
\[
  \Lscr(u)
  = \begin{pmatrix}
      \one_n\T u & (u^{-})\T
   \\ u^{-}  & H^{-1} + \Sigma
    \end{pmatrix}
    \textt{where}
    \Sigma:=\diag(u^{+}),
    \quad
    \begin{array}{c}
      u^{+}:=\phantom-u_{1}+u_{2},\\
      u^{-}:=-u_{1}+u_{2}.
    \end{array}
\]
Systems that involve $\Lscr$ can be solved by pivoting on the block
$(H\inv+\Sigma)$. The cases where this approach is efficient are
exactly those that are efficient in the case of
Example~\ref{L1_Example-cont}.
\end{example}

\begin{example}[2-norm; cf.\@ Example~\ref{ex:2-norm}]
  \label{ex:group-lasso-cont}

  Example~\ref{ex:2-norm} gives the QS representation for the 2-norm
  function. Because $\Kscr=Q^n$, then $\block(u)=(2uu^T-[u\T Ju]J)^2$,
  where $u = (u_0, \bar u)$ and $J$ is specified
  in~\eqref{eq:block-diag-def}. With the expressions for $A$ and $B$
  from Example~\ref{ex:2-norm}, the linear operator $\Lscr$ reduces to
  \begin{equation} \label{eq:lin-op-sum-of-squares}
    \Lscr(u) = H\inv+\alpha I_n + vv^T,
    \text{with}
    \alpha = (u\T J u)^2,\ v=\sqrt{8u_0}\cdot\bar u.
  \end{equation}
  This amounts to a perturbation of $H\inv$ by a multiple of the
  identity, followed by a rank-1 update. Therefore, systems that
  involve $\Lscr$ can be solved at the cost of solving systems with
  $H + \alpha I_n$ (for some scalar $\alpha$).

  Of course, the proximal map of the 2-norm is easily computed by
  other means; our purpose here is to use this as a building block for
  more useful penalties, such as Example~\ref{ex:sum-of-norms}, which
  involves the sum-of-norms function shown in
  \eqref{eq:sum-of-norms-g}.  Suppose that the $p$ partitions do not
  overlap, and have size $n_i$ for $i=1,\ldots,p$.  The operator
  $\Lscr$ in~\eqref{eq:lin-op-sum-of-squares} generalizes to
  \begin{equation*}
    \Lscr(u) = H^{-1}+
    \underset{W}{\underbrace{
    \begin{pmatrix}
      \alpha_{1}I_{n_1}+v^{1}(v^{1})^{T}\\
     & \ddots\\
     &  & \alpha_{p}I_{n_p}+v^{p}(v^{p})^{T}
    \end{pmatrix}
    }}
    \ ,\quad
    \begin{aligned}
      u^{i}  & =(u^i_{0},\bar{u}^{i})\\
      \alpha_{i}  &= (u^{iT} Ju^{i})^2\\
      v^{i} & =\sqrt{8u^i_{0}}\cdot \bar{u}^{i},
    \end{aligned}
  \end{equation*}
  where each vector $u_i$ has size $n_i+1$.

  When $p$ is large, we can treat each diagonal block of $W$ as an
  individual (small) diagonal-plus-rank-1 matrix. If $H\inv$ is
  diagonal-plus-low-rank, for example, the diagonal part of $H\inv$
  can be subsumed into $W$. In that case, each diagonal block in $W$
  remains diagonal-plus-rank-1, which can be inverted in parallel by
  handling each block individually. Subsequently, the inverse of
  $\Lscr$ can be obtained by a second correction.

  Another approach, when $p$ is small, is to consider $W$ as a
  diagonal-plus-rank-$p$ matrix:
  \[
    W =
    \begin{pmatrix}
    \alpha_{1}I_{n_1}\\
     & \ddots\\
     &  & \alpha_{p}I_{n_p}
    \end{pmatrix}
    +\begin{pmatrix}
      v^{1}\\ & \ddots\\ &  & v^{p}
    \end{pmatrix}
    \begin{pmatrix}
      v^{1}\\
      & \ddots\\
      &  & v^{p}
    \end{pmatrix}^T.
  \]
  This representation is convenient: systems involving $\Lscr$
  can be solved efficiently in a manner identical to that of
  Example~\ref{L1_Example-cont} because $W$ is a
  diagonal-plus-low-rank matrix.

\end{example}

\begin{example}[separable QS functions]
  \label{ex:separable} Suppose that $g$ is separable, i.e.,
\[
 g(x)=\gamma(x_1) + \cdots + \gamma(x_n),
\]
where $\gamma:\Real\to\Real$ is a QS function with parameters
$(A_\gamma,b_\gamma,B_\gamma,d_\gamma,\Real_+^{np})$, and $p$ is an
integer parameter that depends on $\gamma$. The parameters $A$ and $B$
for $g$ follow from the concatenation
rule~\eqref{Calculus-Rule-Concat}, and $A=(I_n\otimes A_\gamma)$ and
$B=(I_n\otimes B_\gamma)$. Thus, the linear operator $\Lscr$ is given
by
\begin{align*}
  \Lscr(u) =
  (I_n \otimes B_\gamma)H^{-1}(I_n \otimes B_\gamma)^{T}
  + (I_n \otimes A_\gamma)^{T}\diag(u)(I_n\otimes A_\gamma).
\end{align*}
Apply the SW identity to obtain
\[
\Lscr(u)^{-1} =
\Lambda^{-1}-\Lambda^{-1}(I_n\otimes B_\gamma)(H+\Sigma)^{-1}(I_n\otimes B_\gamma)^{T}\Lambda^{-1},
\]
where $\Lambda=\diag(\Lambda_1,\ldots,\Lambda_n)$,
\begin{align*}
  \Lambda_i =
  A_\gamma\T\diag(u^i)A_\gamma,
  \quad\text{and}\quad
  \Sigma = \diag(B_\gamma^{T}\Lambda_1^{-1}B_\gamma,\dots,B_\gamma^{T}\Lambda_n^{-1}B_\gamma).
\end{align*}
Because the function $\gamma$ takes a scalar input, $B_\gamma$ is a
vector. Hence $\Sigma$ is a diagonal matrix. Note too that $\Lambda$
is a block diagonal matrix with $n$ blocks each of size $p$. We can
then solve the system $\Lscr(u)p = q$ with the following steps:

\begin{algorithm}[H]
  \SetKwInOut{Input}{input}\SetKwInOut{Output}{output}
  \DontPrintSemicolon \nl
  $q_1 \gets (I_n \otimes B_\gamma)^T\Lambda^{-1}q$ \; \nl
  $q_2 \gets (H + \Sigma)^{-1}q_1$ \; \nl
  $q_3 \gets \Lambda^{-1}q_2 -\Lambda^{-1}(I_n\otimes B_\gamma)q_2$
\end{algorithm}

The cost of solving systems with the operator $\Lscr$ is dominated by
solves with the block diagonal matrix $\Lambda$ (Steps~1 and~3) and
$H + \Sigma$ (Step~2). The cost of the latter linear solve is explored
in Example~\ref{L1_Example-cont}.
\end{example}

\section{A proximal quasi-Newton method}
\label{sec:quasi-Newton}

We now turn to the proximal-gradient method discussed in \S\ref{introduction}. Our primary goal is to demonstrate the feasibility of the interior approach for evaluating proximal operators of QS functions. A secondary goal is to illustrate how this technique leads to an efficient extension of the quasi-Newton method for nonsmooth problems of practical interest.

We follow \citet{Scheinberg2016} and implement a limited-memory BFGS
(L-BFGS) variant of the proximal-gradient method that has no
linesearch and that approximately evaluates the proximal
operator. Scheinberg and Tang establish a sublinear rate of
convergence for this method when the Hessian approximations are
suitably modified by adding a scaled identity matrix, and when the
scaled proximal maps are evaluated with increasing accuracy. In their
proposal, the accuracy of the proximal evaluation is based on bounding
the value of the approximation to \eqref{eq:4}. We depart from this
criterion, however, and instead use the
residual~\eqref{eq:int_residual} obtained by the interior solver to
determine the required accuracy. In particular, we require that the
optimality criterion of the interior algorithm used to evaluate the
operator is a small multiplicative constant $\kappa$ of the current
optimality of the outer proximal-gradient iterate, i.e.,
\[
\|R_\mu(y,v,s)\| \leq \kappa \norm{x_k - \prox{}{g}(x_k-\nabla
  f(x_k))}.
\]
This heuristic is reminiscent of the accuracy required of the linear
solves used by an inexact Newton method for root
finding~\cite{DembEiseStei:1982}.  Note that the proximal map
$\prox{}{g}\equiv\prox{I}{g}$ used above is unscaled, which in many
cases can be easily computed when $g$ is seperable.

\subsection{Limited-memory BFGS updates} \label{sec:L-BFGS}

Here we give a brief outline the L-BFGS method for
obtaining Hessian approximations of a smooth function $f$. We follow
the notation of \citet[\S6.1]{NoW99}, who use $H\k$ to denote the
current approximation to the \emph{inverse} of the Hessian of $f$. Let
$x\k$ and $x\kp1$ be two consecutive iterates, and define the vectors
\[
  s_{k} = x_{k+1}-x_{k},
  \quad\text{and}\quad
  y_{k} = \nabla f(x_{k+1})-\nabla f(x_{k}).
\]
A ``full memory'' BFGS method updates the approximation $H_k$ via
the recursion
\begin{align*}
  H_{0}=\sigma I,
  \qquad
  H_{k+1} = H_{k}-\frac{H_{k}s_{k}s_{k}^{T}H_{k}}{s_{k}^{T}H_k s_k}+\frac{y_{k}y_{k}^{T}}{y_{k}^{T}s_{k}},
\end{align*}
for some positive parameter $\sigma$ that defines the initial
approximation. The limited-memory variant of the BFGS update (L-BFGS)
maintains the most recent $m$ pairs $(s\k,y\k)$, discarding older
vectors. In all cases, $m\ll n$, e.g., $m = 10$. The globalization
strategy advocated by \citet{Scheinberg2016} may add a small multiple
of the identity to $H_k$. This modification takes the place of a
potentially expensive linesearch, and the correction is increased at
each iteration if a certain condition for decrease is not satisfied.

Each interior iteration for evaluating the proximal operator depends
on solving linear systems with $\Lscr$ in~\eqref{eq:L-def}. In all of
the experiments presented below, each interior iteration has a cost
that is linear in the number of variables $n$.

\section{Numerical experiments}

We have implemented the proximal quasi-Newton method as a Julia
package \citep{BEKS14}, called \Qsip, designed for problems of the
form~\eqref{eq:6}, where $f$ is smooth and $g$ is a QS function. The
code is available at the URL
\begin{center}
  \url{https://github.com/MPF-Optimization-Laboratory/QSip.jl}
\end{center}
A primal-dual interior method, based on ideas from the CVXOPT software
package \citep{Andersen:2010}, is used for
Algorithm~\ref{algo:proxAlgo}. We consider below several examples. The
first three examples apply the \Qsip\ solver to minimize benchmark
least-squares problems with different nonsmooth regularizers that are
QS representable; the last example applies the solver to a sparse
logistic-regression problem on a standard data set.

\subsection{Timing the proximal operator} \label{sec:timeprox}

The examples that we explored in \S\ref{sec:eval-prox-oper} have a
favorable structure that allows each interior iteration for evaluating
the proximal map $\prox{H}{g}(x)$ to scale linearly with problem
size. In this section we verify this behavior empirically for problems
with the structure
\begin{equation} \label{eq:proxtime}
H = I + UU^T, \qquad g(x) = \|x\|_1, \qquad U \in \Real^{n \times k}
\end{equation}
for different values of $k$ and $n$.  This choice of
diagonal-plus-low-rank matrices is designed to mimic the structure of
matrices that appear in L-BFGS. Here $U$ and $x$ are chosen with
random normal entries. As described in Example~\ref{L1_Example}, the
system $\mathcal{L}(u)$ is inverted in linear time using the SW
identity.

We evaluate the proximal map on $100$ random instances for each
combination of $k$ and $n$, and plot in Figure 1 the average time
needed to reach an accuracy of $10^{-7}$, as measured by the
optimality conditions in the interior algorithm.  Because in practice
the number of iterations of the interior method is almost independent
of the size of the problem, the time taken to compute the proximal map
is a predictable, linear function of the size of the problem.

\begin{figure}
\centering
\includegraphics[scale = 0.5]{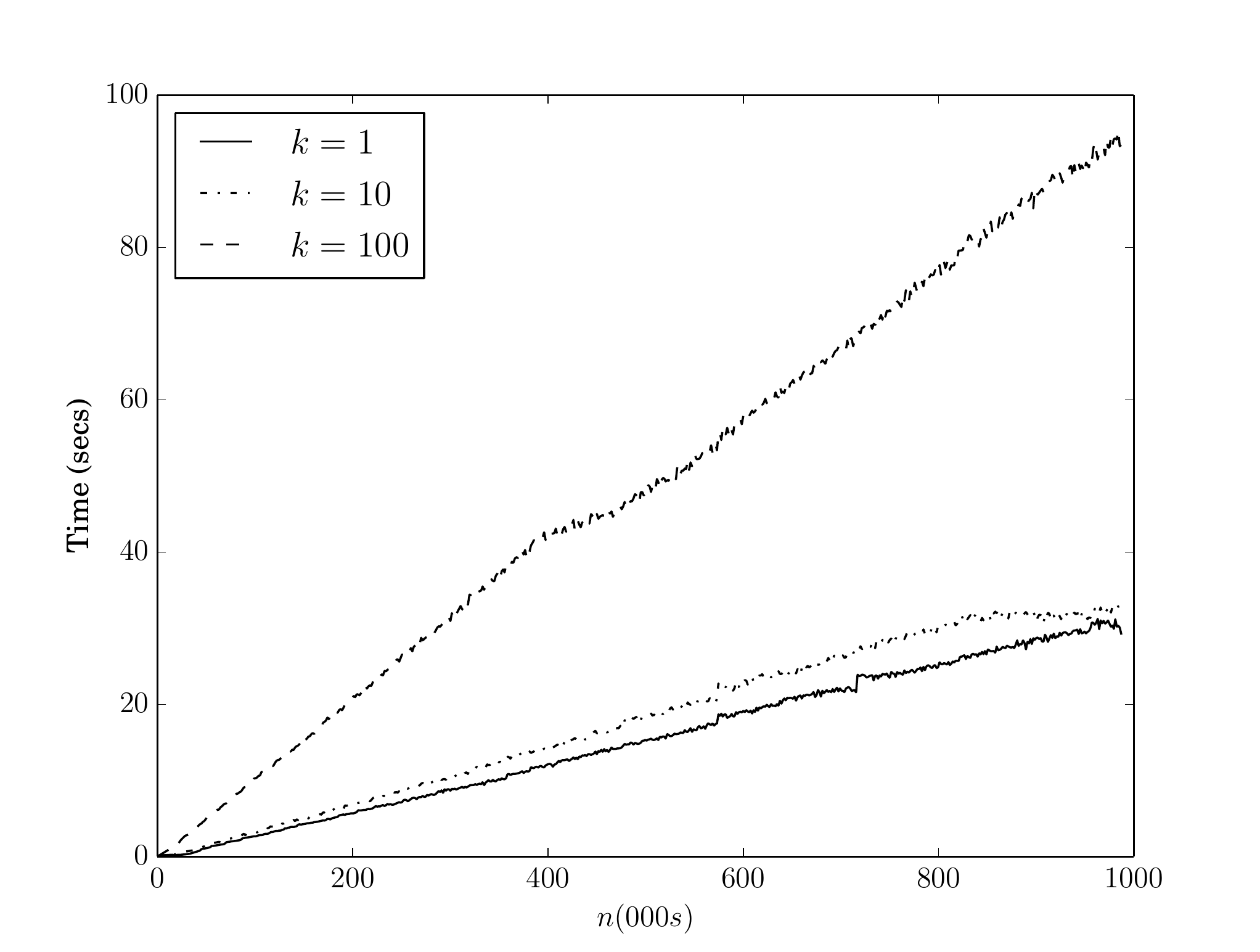} 
\caption{Time taken to compute $\prox{H}{g}(x)$ versus $n$, for
  $k=1, 10, 100$; see~\eqref{eq:proxtime}.}
\end{figure}

\subsection{Synthetic least-square problems}

The next set of examples all involve the least-squares objective
\begin{equation}\label{eq:ls-objective}
  f(x) = \half\norm{Ax-b}^2_2.
\end{equation}
Two different procedures are used to construct matrices $A$, as
described in the following sections. In all cases, we follow the
testing approach described by \citet{6399612} for constructing a test
problem with a known solution: fix a vector $x^{\star}$ and choose
$b=Ax^{\star}-A^{-T}v$, where $v\in\partial g(x^\star)$. Note that
\[
\partial(f+g)(x^\star)=A^{T}(Ax^\star-[Ax^{\star}-A^{-T}v])+\partial g(x^{\star})=\partial g(x^{\star})-v.
\]
Because $v\in \partial g(x^\star)$, the above implies that
$0\in\partial(f+g)(x^\star)$, and hence $x^\star$ minimizes the
objective $f+g$.  In the next three sections, we apply \Qsip\ in turn
to problems with $g$ equal to the 1-norm, the group LASSO (i.e., sum
of 2-norm functions), and total variation.

\subsubsection{One-norm regularization} \label{sec:1-norm-reg}

In this experiment we choose $g = \|\cdot\|_1$, which gives the 1-norm
regularized least-squares problem, often used in applications of
sparse optimization. Following the details in Example
\ref{L1_Example}, the system $\mathcal{L}(u)$ is a
diagonal-plus-low-rank matrix, which we invert using the SW identity.

The matrix $A$ in~\eqref{eq:ls-objective} is a 2000-by-2000 lower
triangular matrix with all nonzero entries equal to $1$. The bandwidth
$p$ of $A$ is adjustable, and determines its coherence
\[
\mbox{coherence}(A)
=
\max_{i\neq j}\frac{a_i\T a_j}{\|a_i\|\|a_j\|} = \sqrt{\frac{p-1}{p}},
\]
where $a_i$ is the $i$th column. As observed by \citet{6399612}, the
difficulty of 1-norm regularized least-squares problems are strongly
influenced by the coherence. Our experiments use matrices $A$ with
bandwidth $p=500, 1000, 2000$.

Figure~\ref{Exp_L1} shows the results of applying the \Qsip\ solver
with a memories $k=1,10$, labeled ``QSIP mem $= k$''. We also consider
comparisons against two competitive proximal-based methods. The first
is a proximal-gradient algorithm that uses the Barzilai-Borwein
steplength \cite{BarzBorw:1988,wright2009sparse}.  This is our own
implementation of the method, and is labeled ``Barzilai-Borwein'' in
the figures. The second is the proximal quasi-Newton method
implemented by \citet{NIPS2012_4523}, which is based on a
symmetric-rank-1 Hessian approximation; this code is labeled
``PG-SR1''. The \Qsip\ solver with memory of 10 outperforms the other
solvers. The quasi-Newton approximation benefits problems with high
coherence ($p$ large) more than problems with low coherence ($p$
small). In all cases, the experiments reveal that the additional cost
involved in evaluating a proximal operator (via an interior method) is
balanced by the overall cost of the algorithm, both in terms of
iterations (i.e., matrix-vector products with $A$) and time.

\begin{figure}[t]
  \label{Exp_L1}
  \centering
  \begin{tabular}{@{}c@{\ }c@{}}
   \includegraphics[width=0.485\textwidth]{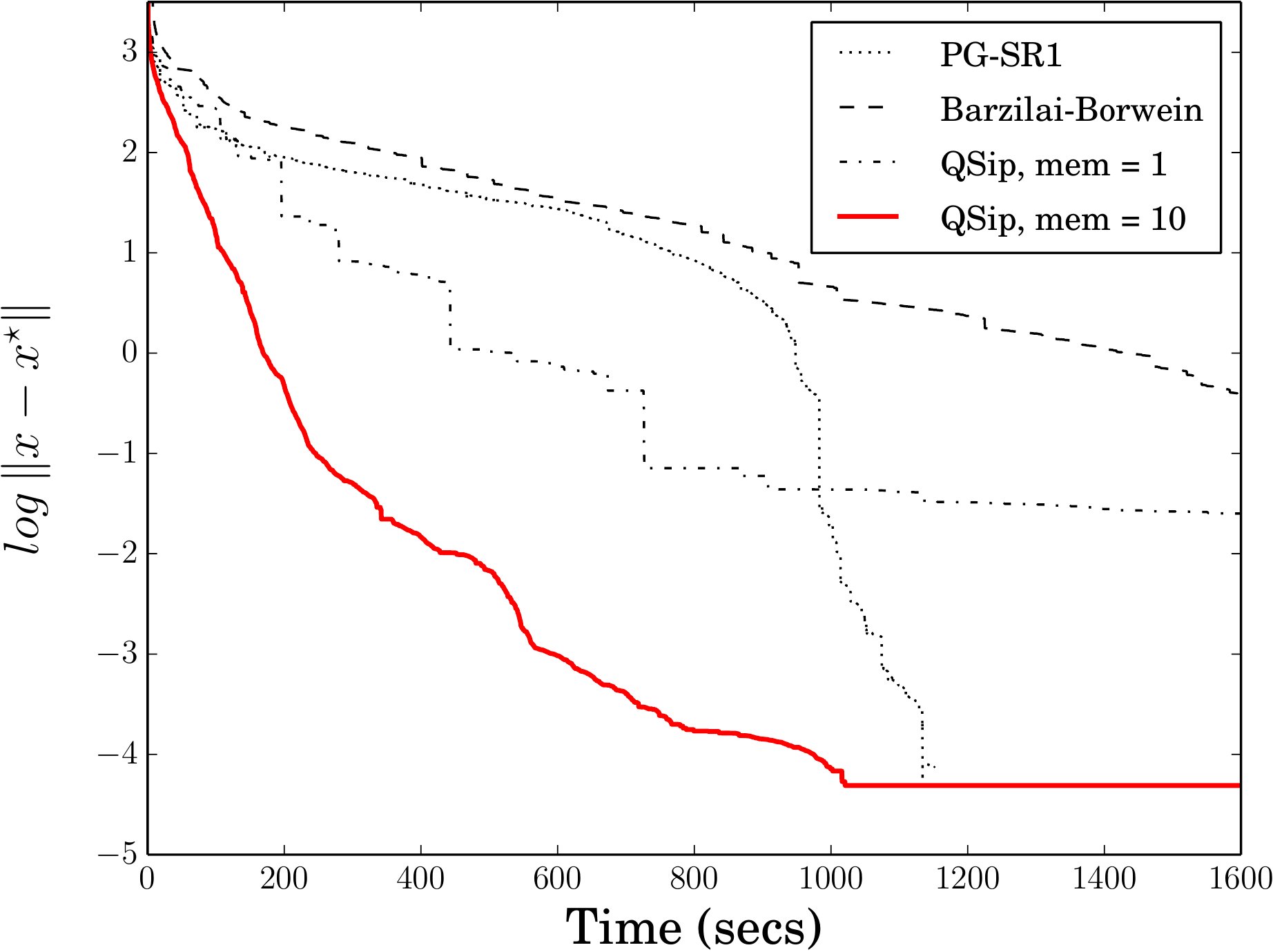}
  &\includegraphics[width=0.50\textwidth]{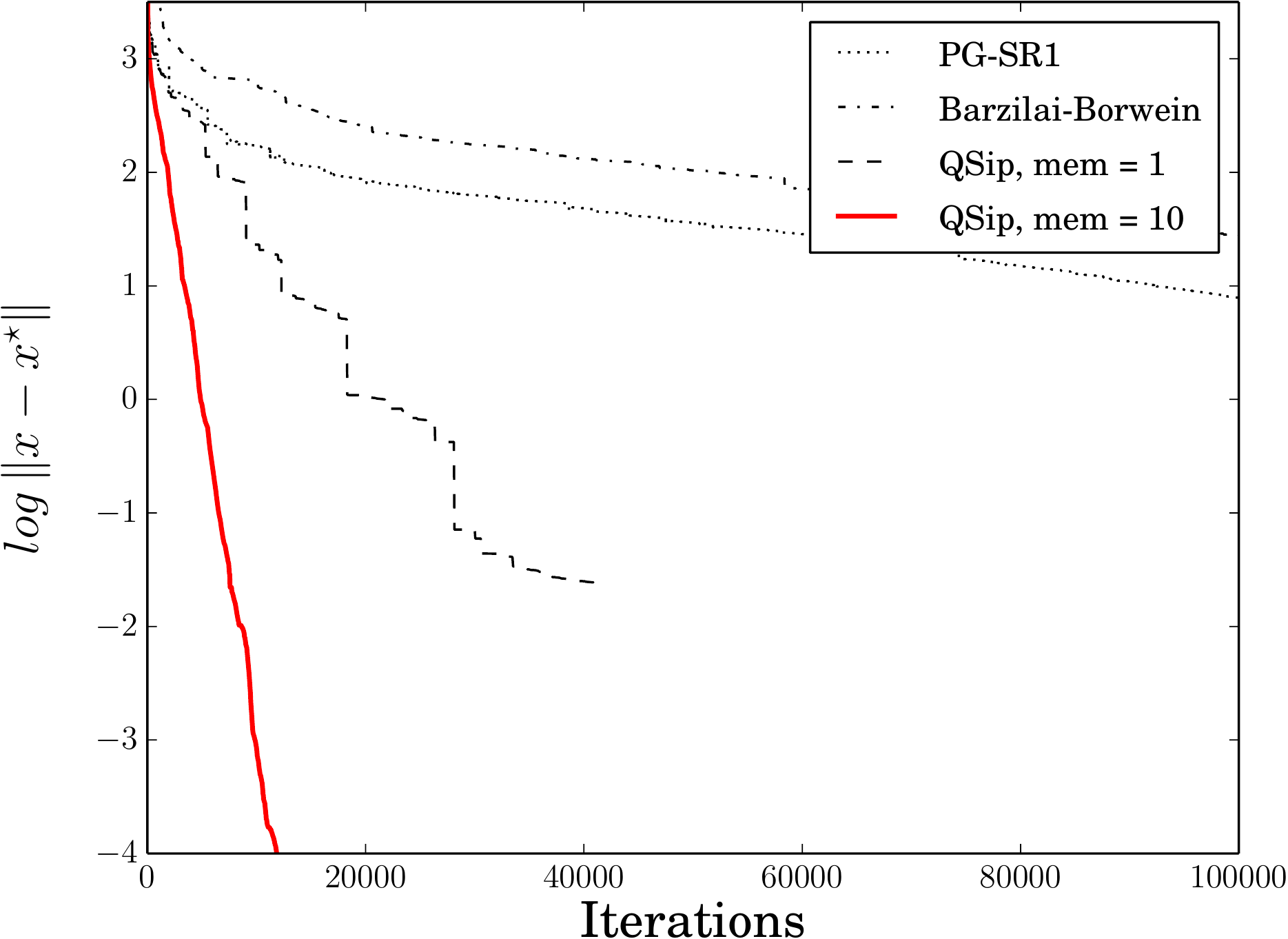}
 \\\includegraphics[width=0.485\textwidth]{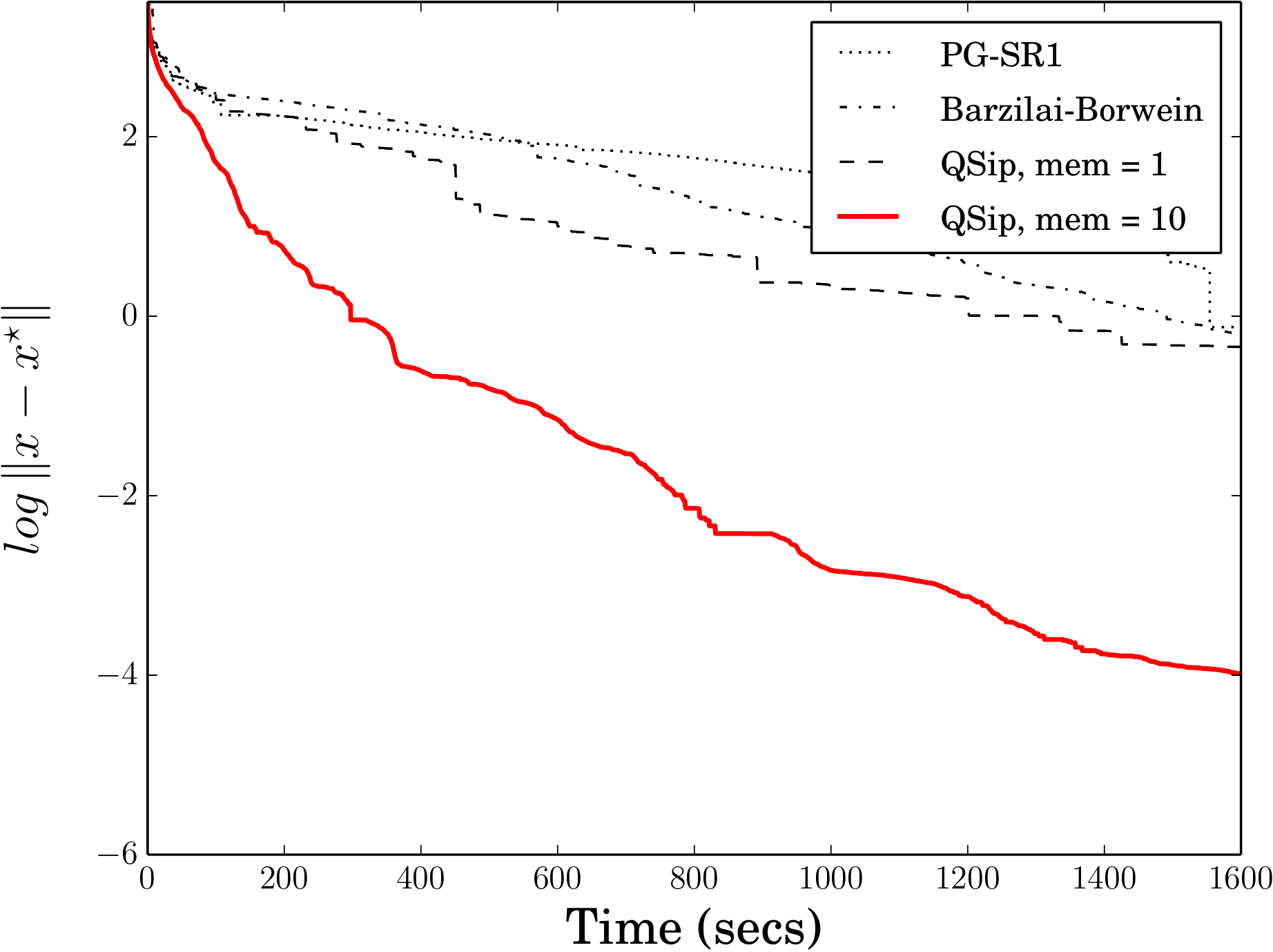}
  &\includegraphics[width=0.50\textwidth]{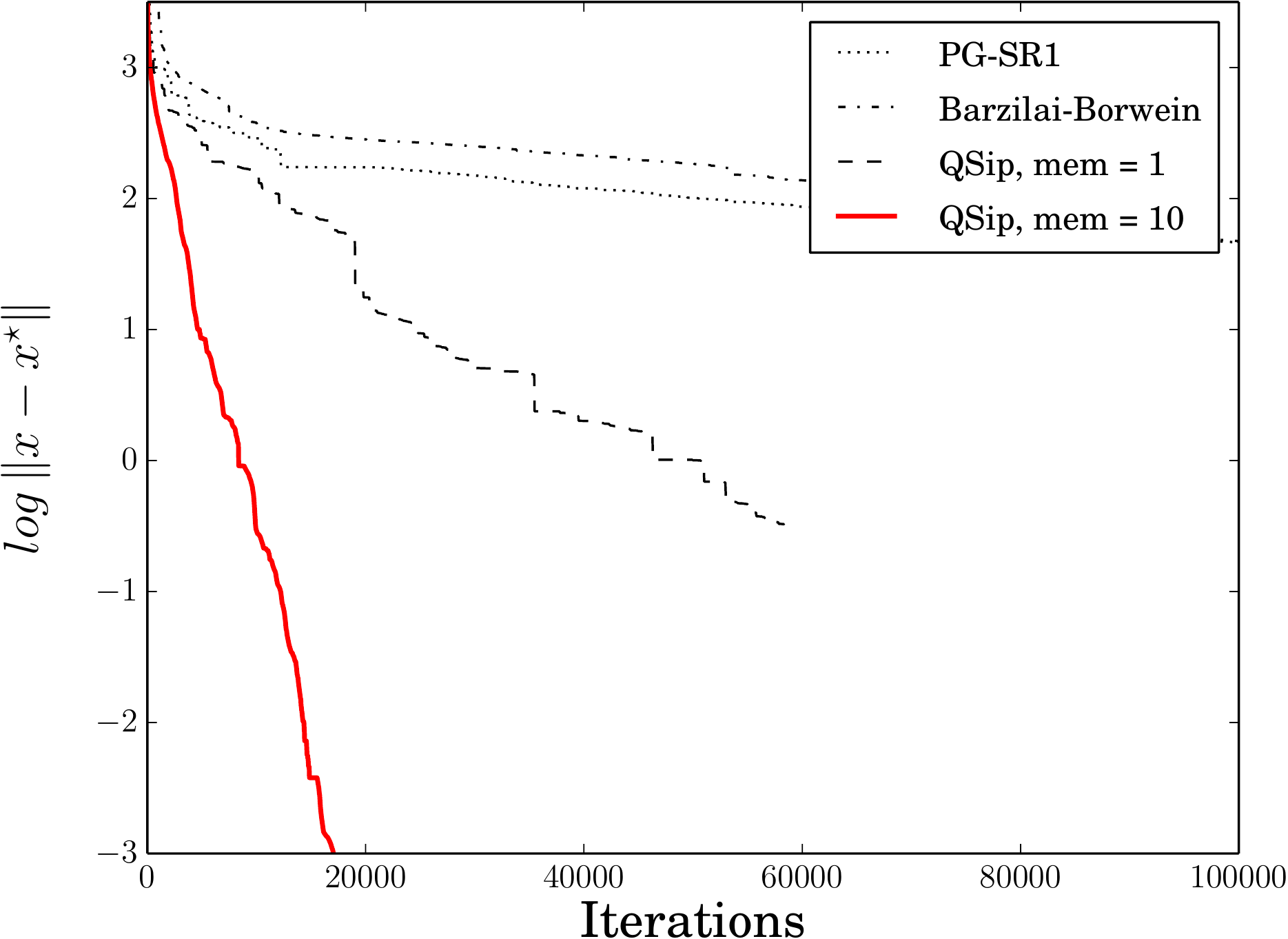}
 \\\includegraphics[width=0.485\textwidth]{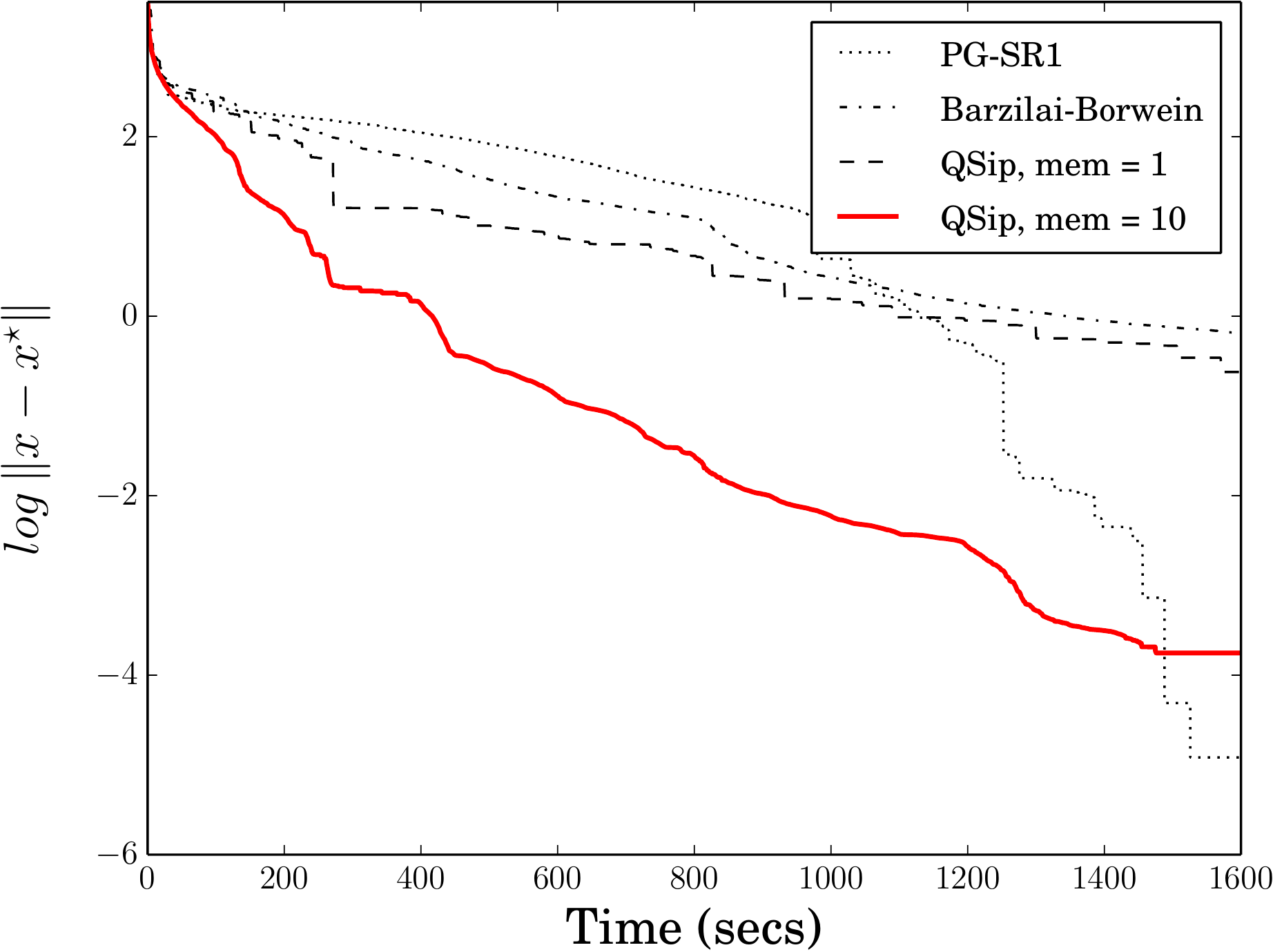}
  &\includegraphics[width=0.50\textwidth]{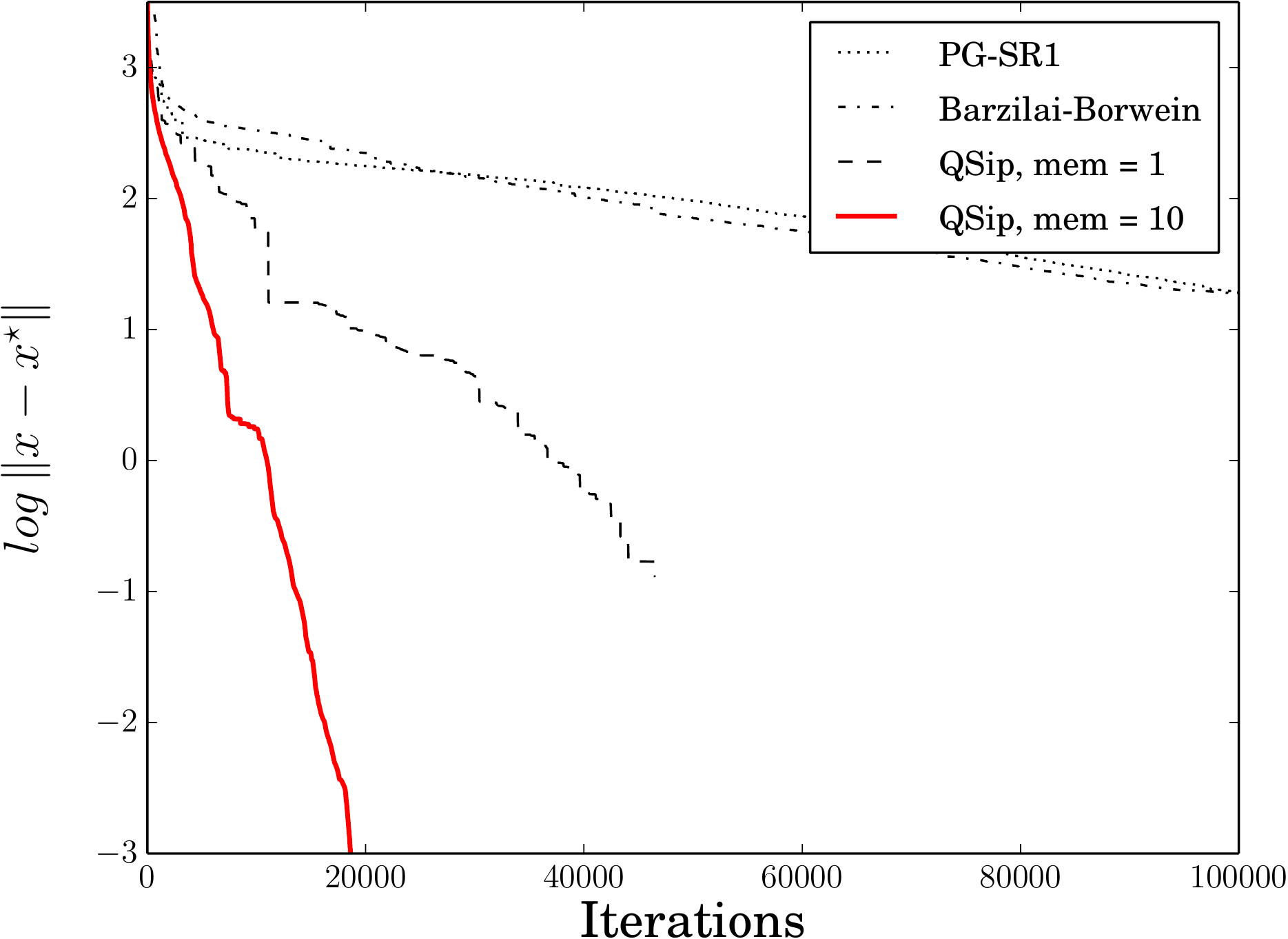}
  \end{tabular}
  \caption{Performance of solvers applied to 1-norm regularized
    least-squares problems of increasing difficulty. The left and
    right columns, respectively, track the distance of the current
    solution estimate to the true solution versus time and iteration
    number.  Top row: $p=2000$ (highest coherence); middle row:
    $p=1000$; bottom row: $p=500$ (lowest coherence). }

\end{figure}

\subsubsection{The effect of conditioning}
\label{sec:conditioning}

It is well known that the proximal-gradient method converges slowly
for ill conditioned problems. The proximal L-BFGS method may help to
improve convergence in such situations. We investigate the observed
convergence rate of the proximal L-BFGS approach on a family of
least-squares problems with 1-norm regularization with varying degrees
of ill conditioning. For these experiments, we take $A$
in~\eqref{eq:ls-objective} as the 2000-by-2000 matrix
\[
A = \alpha_L \begin{pmatrix}
 T & 0\\
0 & 0
\end{pmatrix} + \alpha_\mu I,
\]
where $T$ is a 1000-by-1000 tridiagonal matrix with constant diagonal
entries equal to 2, and constant sub- and super-diagonal entries equal
to ${-}1$. The parameter $\alpha\L/\alpha_\mu$ controls the
conditioning of $A$, and hence the conditioning of the Hessian $A\T A$
of $f$.

We run L-BFGS with 4 different memories (``mem''): $0$ (i.e., proximal
gradient with a Barzilai-Borwein steplength), $1$, $10$, and $100$.
We terminate the algorithm either when the error drops beneath
$10^{-8}$, or the method reaches $10^3$ iterations. Our method of
measuring the {observed convergence} (OC) computes the line of best
fit to the log of optimality versus $k$, which results in the quantity
\[
\mbox{Observed Convergence}
:= \frac{\sum_{k=0}^{N}k\cdot\log \|x_{k}-x_{*}\|}{\sum_{k=0}^{N} \log \|x_{k}-x_{*}\|},
\]
where $N$ is the total number of iterations.

The plot in Figure~\ref{fig:convergence} shows the ratio of the OC for
L-BFGS relative to the observed convergence of proximal gradient
(PG). This quantity can be interpreted the amount of work that a
single quasi-Newton step performs relative to the number of PG
iterations. The plot reveals that the quasi-Newton method is faster at
all condition numbers, but is especially effective for problems with
moderate conditioning. Also, using a higher quasi-Newton memory almost
always lowers the number of iterations. This benefit is most
pronounced when the problem conditioning is poor.

Together with \S\ref{sec:timeprox}, this section gives a broad picture
of the trade-off between the proximal quasi-Newton and proximal
gradient methods. The time required for each proximal gradient
iteration is dominated by the cost of the gradient computation because
the evaluation of the unscaled proximal operator is often trivial. On
the other hand, the proximal quasi-Newton iteration additionally
requires evaluating the scaled proximal operator.  Therefore, the
proximal quasi-Newton method is most appropriate when this cost is
small relative to the gradient evaluation.

\begin{figure}
\centering
\includegraphics[width=0.7\textwidth]{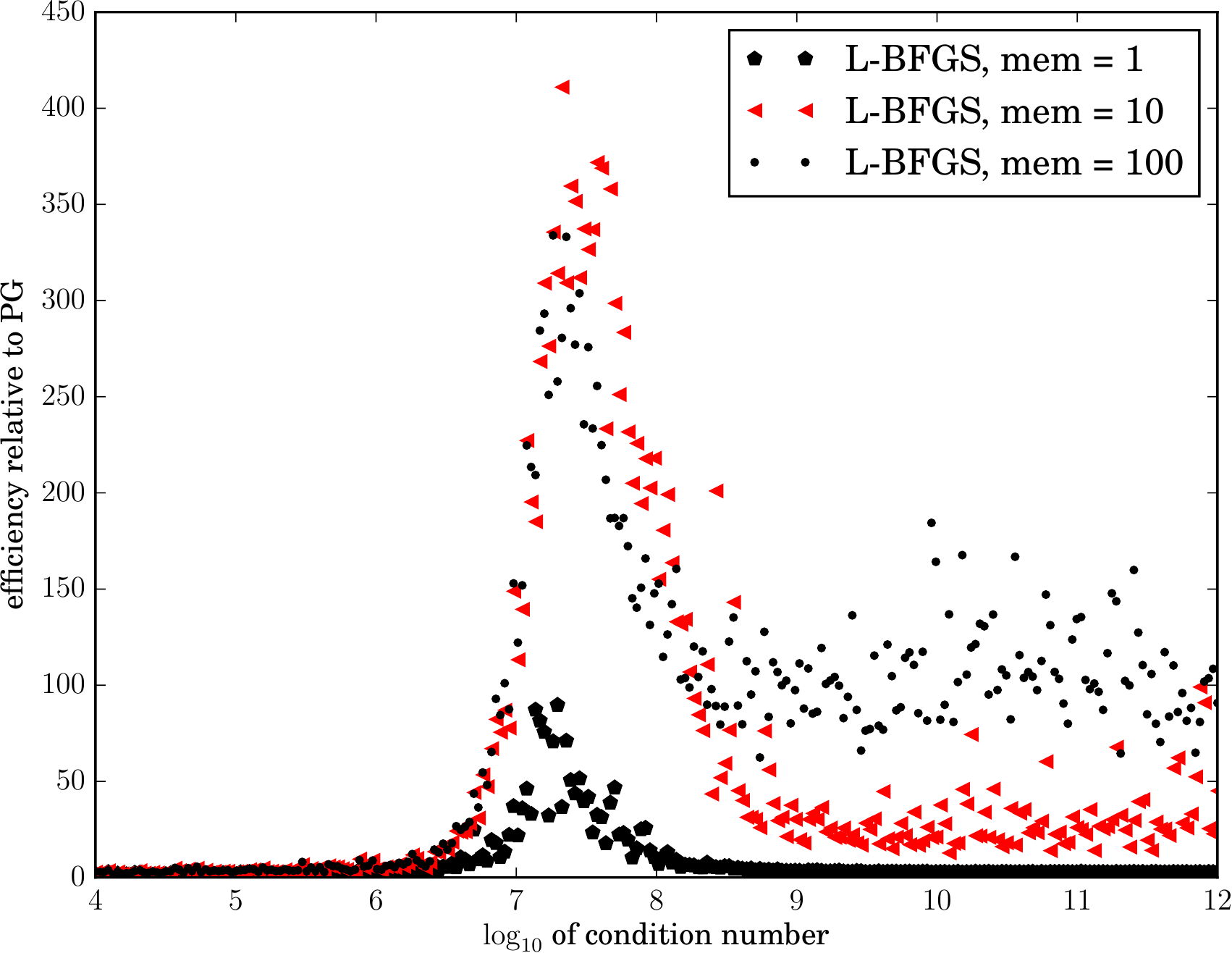}
\caption{Performance of the proximal quasi-Newton method relative to
  proximal gradient for problems of varying condition number.}
\label{fig:convergence}
\end{figure}

\subsubsection{Group LASSO}

Our second experiment is based on the sum-of-norms regularizer
described in Examples~\ref{ex:sum-of-norms}
and~\ref{ex:group-lasso-cont}. In this experiment, the $n$-vector
(with $n=2000$) is partitioned into $p=5$ disjoint blocks of equal
size. The matrix $A$ is fully lower triangular.

Figure~\ref{Exp_L2} clearly shows that the \Qsip\ solver outperforms
the PG method with the Barzilai-Borwein step size.  Although we
required \Qsip\ to exit with a solution estimate accurate within 6
digits (i.e., $\log\norm{x-x^*}\le 10^{-6}$), the interior solver
failed to achieve the requested accuracy because of numerical
instability with the SW formula used for solving the Newton
system. This raises the question of how to use efficient alternatives
to the SW update that are numerically stable and can still leverage
the structure of the problem.
\begin{figure}
  \centering
  \includegraphics[width=0.75\textwidth]{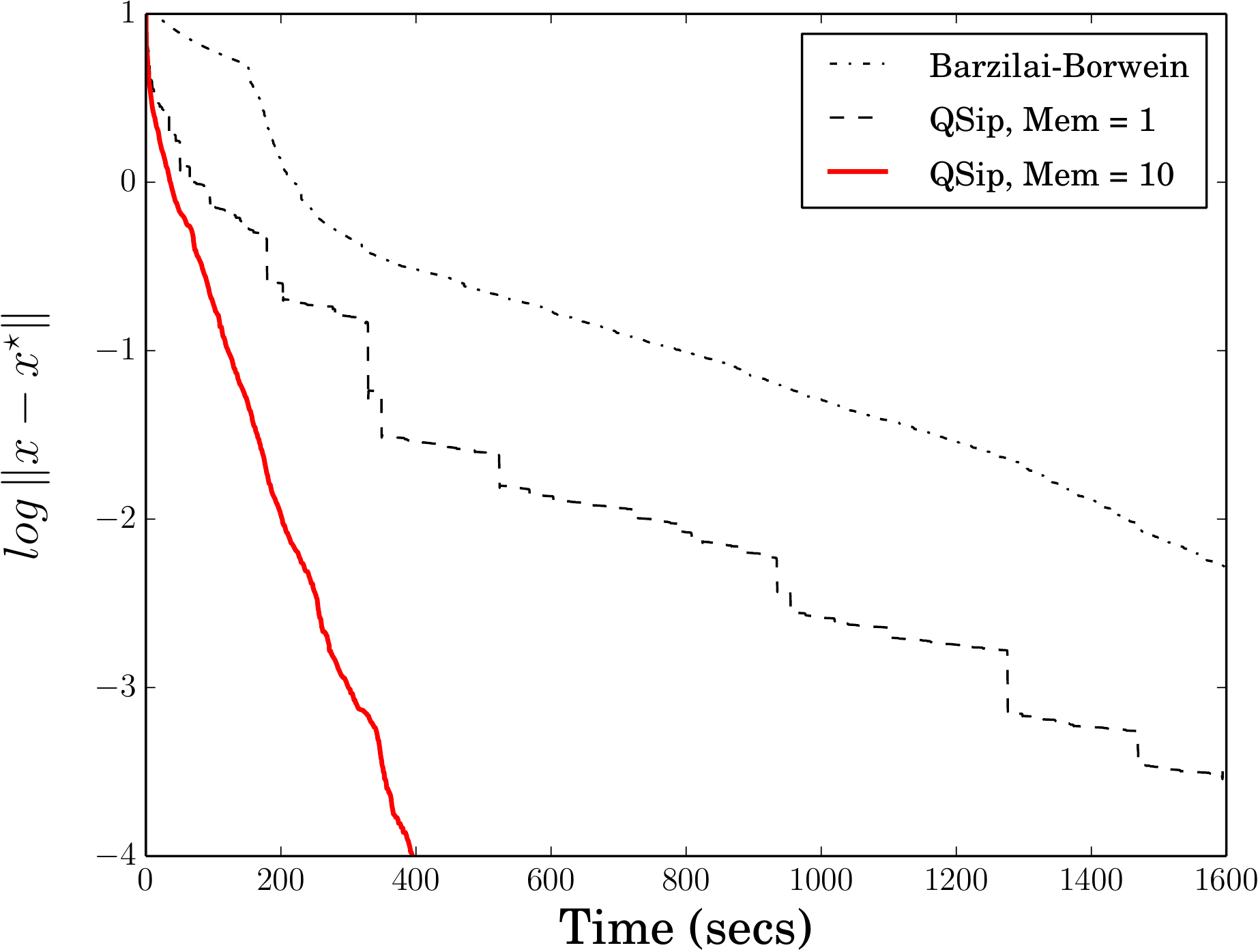}
  \label{Exp_L2}
  \caption{Performance of solvers applied to a group-Lasso problem.
    The horizontal axis measures elapsed time; the vertical axis
    measures distance to the solution.}
\end{figure}

\subsubsection{1-dimensional total variation} \label{sec:TV} Our third
experiment sets
\[
g(x)=\sum_{i=1}^{n-1}|x_{i+1}-x_{i}|,
\]
which is the anisotropic total-variation regularizer described in
Examples~\ref{ex:graph1-norm-example} and~\ref{ex:graph-1-norm}. The
matrix $A$ is fully lower triangular. Figure~\ref{Exp_TV} compares the
convergence behavior of \Qsip\ with the Barzilai-Borwein proximal
solver. The Python package {\tt
  prox-tv}~\cite{conf/icml/Barbero11,barberoTV14} was used for the
evaluation of the (unscaled) proximal operator, needed by the
Barzilai-Borwein solver. The \Qsip\ solver, with memories of 1 and 10,
outperformed the Barzilai-Borwein solver.
\begin{figure}
  \centering
  \includegraphics[scale=0.52]{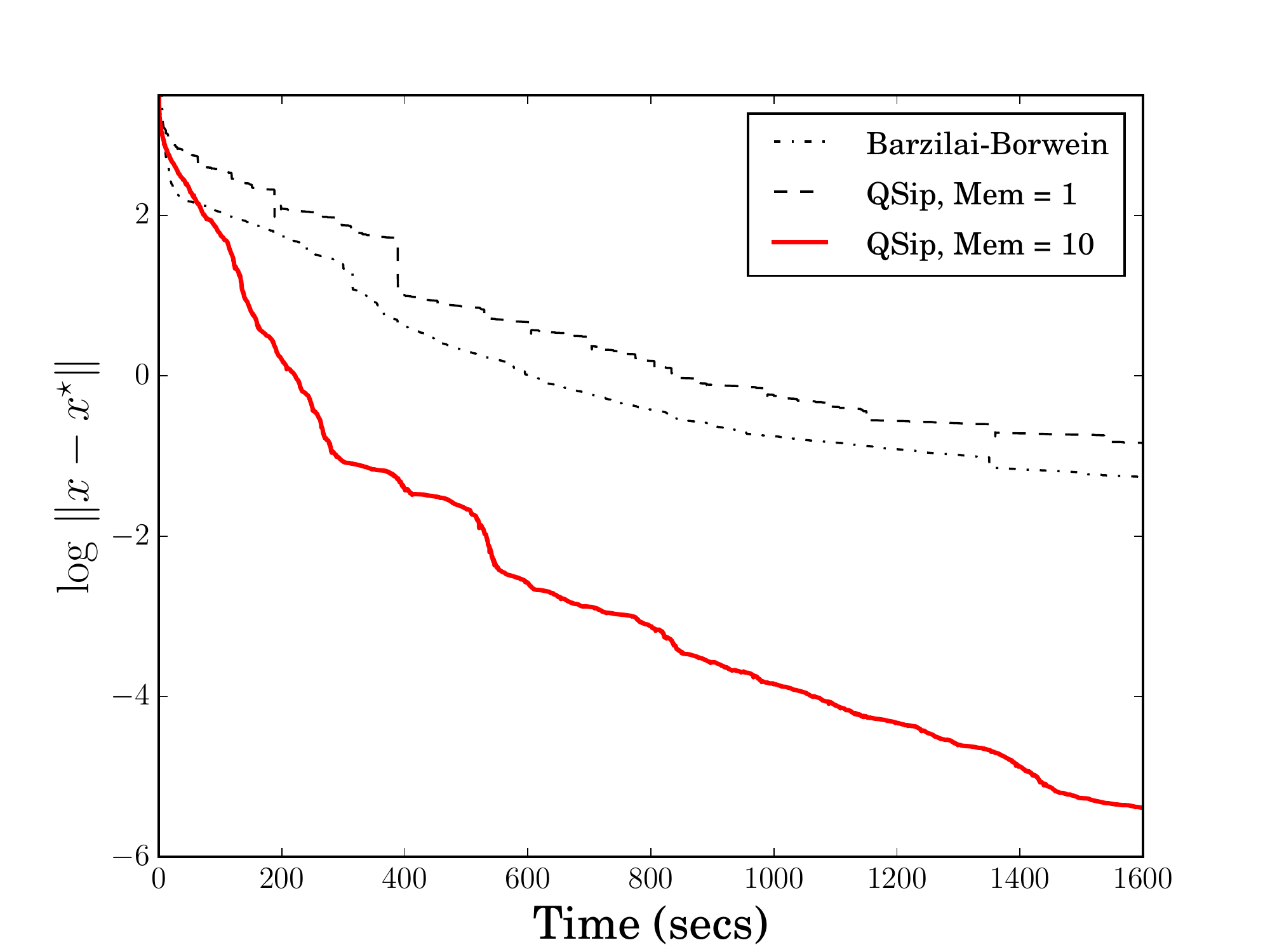}
  \caption{Performance of the \Qsip\ solver applied to a 1-dimensional
    total-variation problem.} \label{Exp_TV}
\end{figure}

\subsection{Sparse logistic regression}

This next experiment tests \Qsip\ on the sparse logistic-regression
problem problem
\[
  \minimize{x} \quad
  \frac{1}{N}\sum_{i=1}^{N}\log(1+\exp[a_{i}\T x])+\lambda \|x\|_{1},
\]
where $N$ is the number of observations.  The
\textit{Gisette}~\cite{guyon2004result} and
\textit{Epsilon}~\cite{PascalChallenge2016} datasets, standard
benchmarks from the UCI Machine Learning
Repository~\cite{Lichman:2013}, are used for the feature vectors
$a_i$.
\textit{Gisette} has $5K$ parameters and $13.5K$ observations;
\textit{Epsilon} has $2K$ parameters with $400K$ observations. These
datasets were chosen for their large size and modest number of
parameters. In all of these experiments, $\lambda = 0.01$.

Figure~\ref{fig:l1logreg} compares \Qsip\ to the Barzilai-Borwein
solver, and to newGLMNet \citep{KimKohLustBoydGori:2007}, a
state-of-the-art solver for sparse logistic regression. (Other
possible comparisons include the implementation of
\citet{Scheinberg2016}, which we do not include because of difficulty
compiling that code.) Because we do not know a priori the solution for
this problem, the vertical axis measures the log of the optimality
residual $\norm{x_k - \prox{}{g}(x_k-\nabla f(x_k))}_\infty$ of the
current iterate. (The norm of this residual necessarily vanishes at
the solution.)  On the \textit{Gisette} dataset, Barzilai-Borwein and
newGLMNNet are significantly faster than the proximal quasi-Newton
implementation. On the \emph{Epsilon} dataset, however, the
quasi-Newton is faster at all levels of accuracy.

\begin{figure}
  \centering
  \begin{tabular}{@{}c@{\ }c@{}}
   \includegraphics[width=0.48\textwidth]{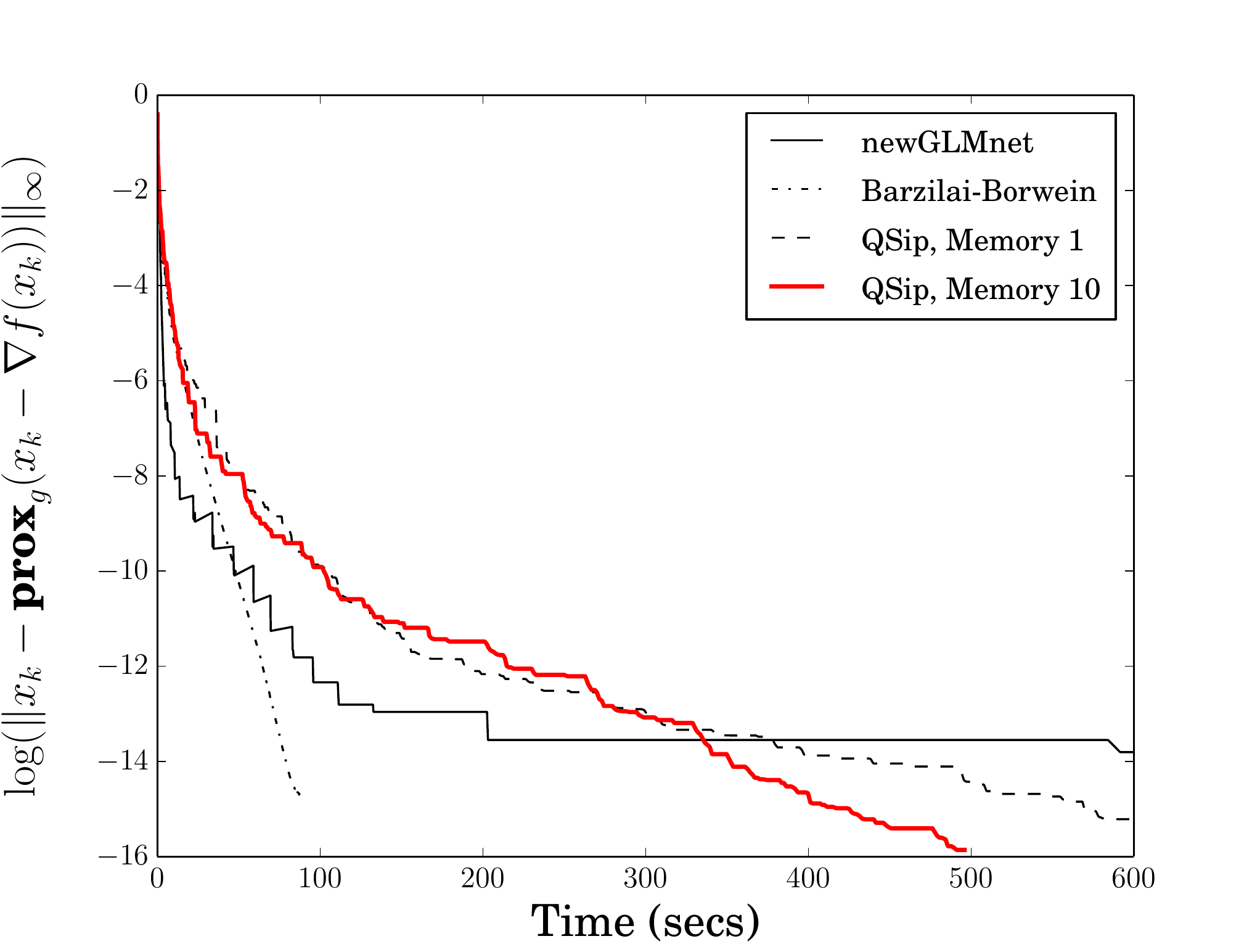}
  &\includegraphics[width=0.48\textwidth]{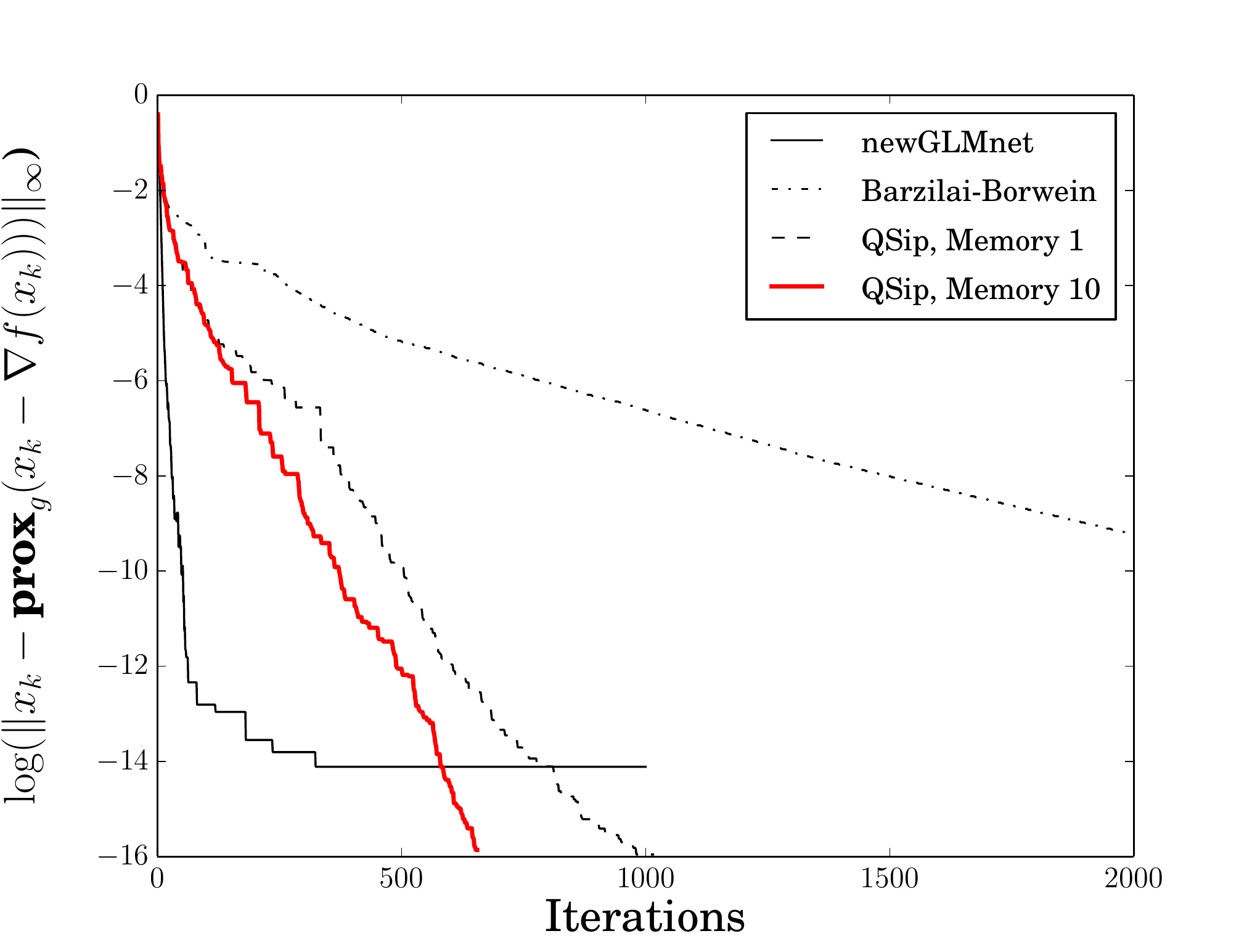}
 \\\includegraphics[width=0.48\textwidth]{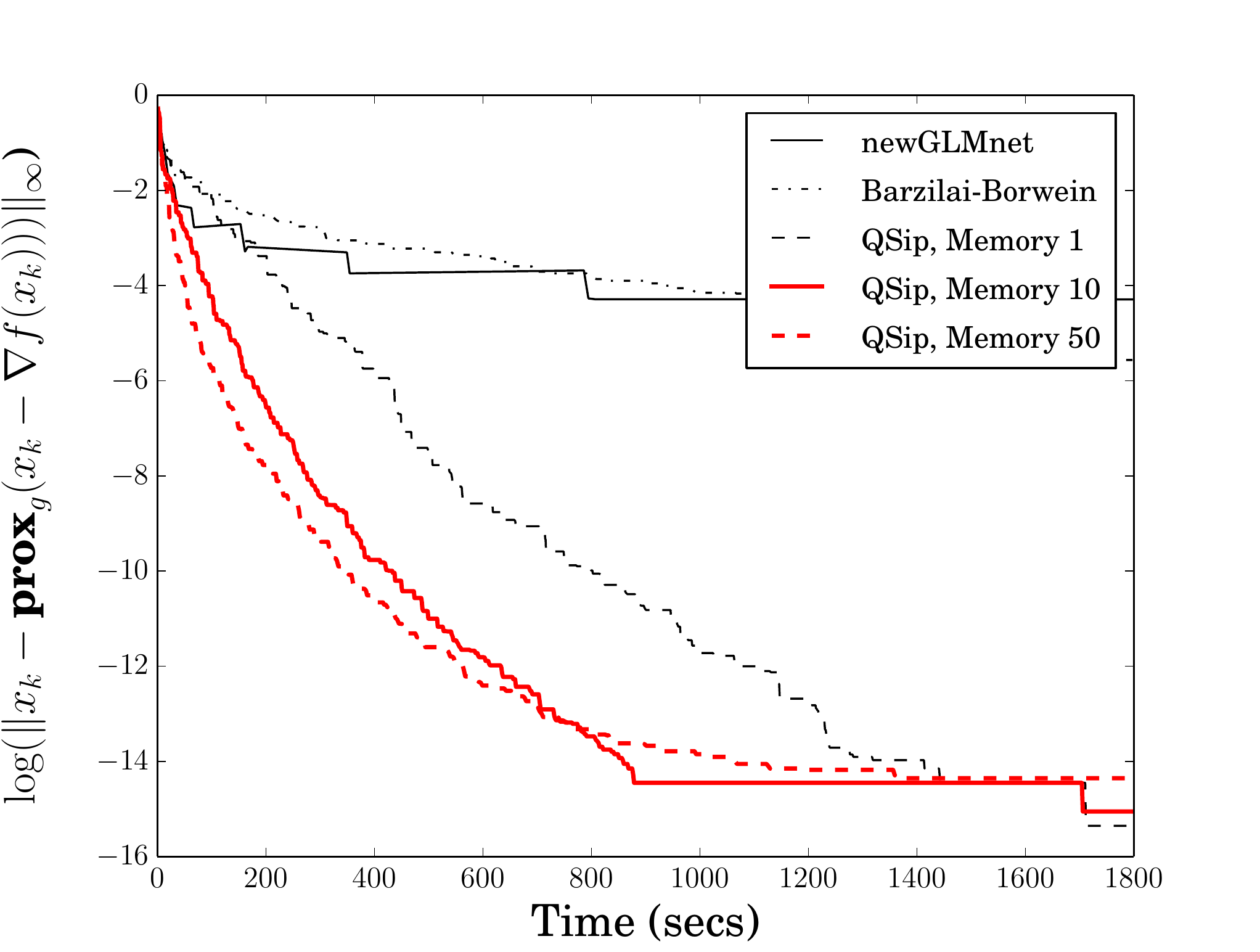}
  &\includegraphics[width=0.48\textwidth]{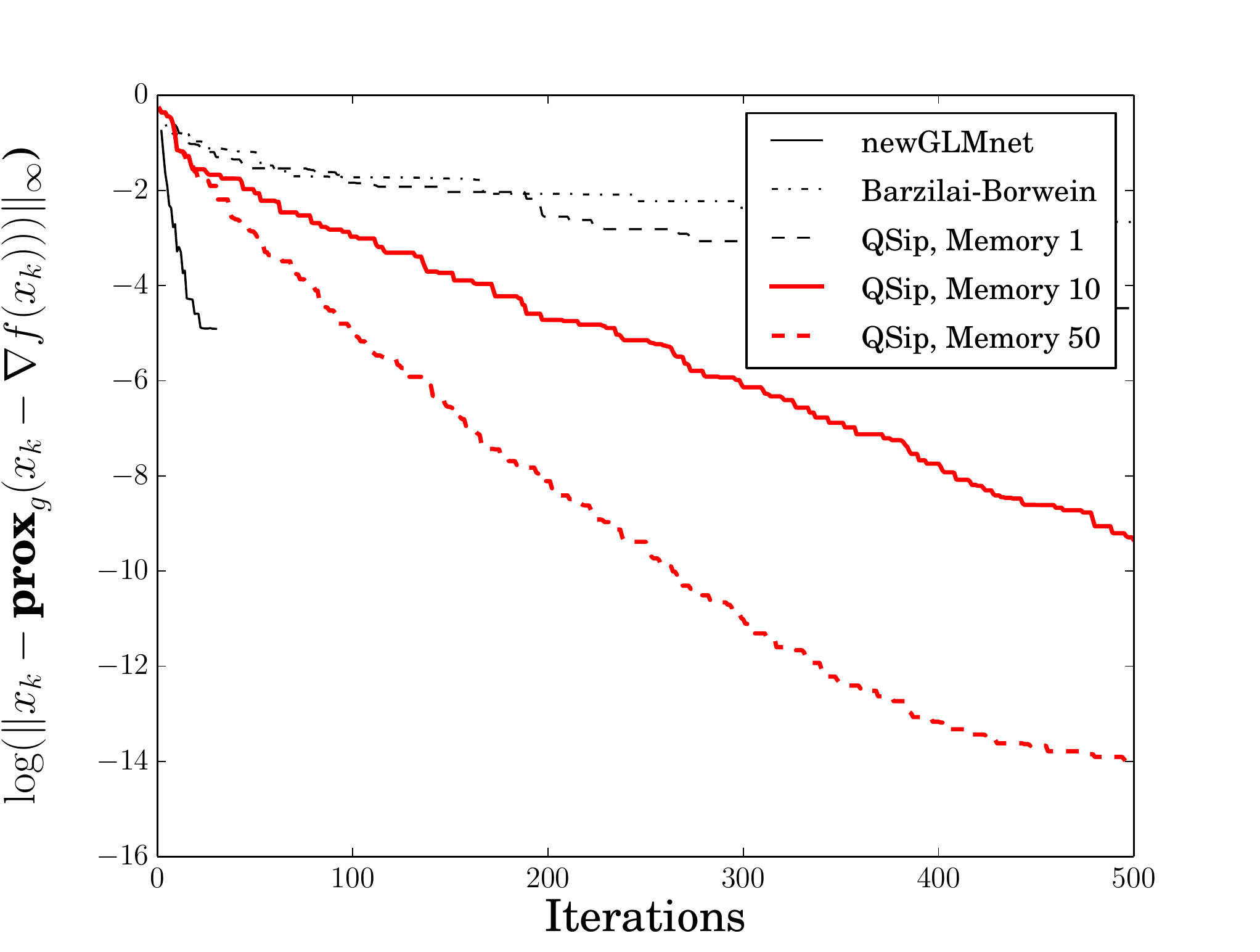}
  \end{tabular}
  \caption{Performance of solvers on a sparse logistic-regression
    problem. Top row: \emph{Gisette} dataset; bottom row:
    \emph{Epsilon} dataset. The left and right columns, respectively,
    track the optimality of the current solution estimate versus
    elapsed time and iteration number.}
  \label{fig:l1logreg}
\end{figure}

\section{Conclusion} \label{sec:conclusion}

Much of our discussion revolves around techniques for solving the
Newton systems~\eqref{eq:5} that arise in the implementation of an
interior method for solving QPs. The Sherman-Woodbury formula features
prominently because it is a convenient vehicle for taking advantage of
the structure of the Hessian approximations and the structured
matrices that typically define QS functions. Other alternatives,
however, may be preferable, depending on the application.

For example, we might choose to reduce the 3-by-3 matrix
in~\eqref{eq:5} to an equivalent symmetrized system
\[
 \pmat{-Q & A^T \\ A & D}\pmat{\Delta y\\\Delta s}
   = -\pmat{-{r_d}\\r_p+V\inv r_\mu}
\]
with $D:=V\inv S$. As described by \citet{benzi2008some}, Krylov-based
method, such as MINRES \citep{PaigSaun:1975}, may be applied to a
preconditioned system, using the preconditioner
\[
  P = \pmat{ -\Lscr(u)  \\ & D },
\]
where $\Lscr(u)$ is defined in~\eqref{eq:9}. This ``ideal''
preconditioner clusters the spectrum into three distinct values, so
that in exact arithmetic, MINRES would converge in three
iterations. The application of the preconditioner requires solving
systems with $\Lscr$ and $D$, and so all of the techniques discussed
in~\S\ref{sec:eval-prox-oper} apply. One benefit, however, which we have
not explored here, is that the preconditioning approach allows us to
approximate $\Lscr\inv(u)$, rather than to compute it exactly, which may
yield computational efficiencies for some problems.

\section*{Acknowledgments}

The authors are grateful to three anonymous referees for their thoughtful suggestions, and for a number of critical corrections. We also wish to thank Fiorella Sgallari for handling this paper as Associate Editor.

\bibliographystyle{abbrvnat}
\bibliography{master}

\appendix

\section{QS Representation for a quadratic} \label{sec:qs-soc}

Here we derive the QS representation of a support function that
includes an explicit quadratic term:
\[
 g(x)=\sup_y\set{y\T(B_0x+d_0)-\half y\T Q y | A_0y\succeq_{\Kscr_0} b_0}.
\]
Let $R$ be such that $R\T R = Q$.  We can then write the
quadratic function in the objective as a constraint its epigraph,
i.e.,
\[
  g(x)
  =\sup_{y,\,t}
  \set{
    y\T(B_{0}x+d_{0})-\half t |
    A_{0}y\succeq_{\Kscr}b_{0},\ \|Ry\|^2 \leq t}.
\]
Next we write the constraint $\|Ry\|^2 \leq t$ as a second-order cone
constraint:
\begin{align*}
  \|Ry\|^{2}\leq t & \iff\|Ry\|^{2}\leq\frac{(t+1)^{2}-(t-1)^{2}}{4}
\\&\iff
   \|Ry\|^{2}+\left(\frac{t-1}{2}\right)^{2}
    \leq\left(\frac{t+1}{2}\right)^{2}
\\
 &\iff
 \sqrt{\|Ry\|^{2}+\left(\frac{t-1}{2}\right)^{2}}\leq\frac{t+1}{2}\\
 &\iff
 \left\|
   \begin{pmatrix}
     0 & \nicefrac{1}{2}\\
     R & 0
   \end{pmatrix}
   \begin{pmatrix}y\\t\end{pmatrix}
  +\begin{pmatrix}-\nicefrac{1}{2}\\0\end{pmatrix}
  \right\|\leq\frac{t+1}{2}\\
 &\iff
   \begin{pmatrix}
     0 & \nicefrac{1}{2}\\
     0 & \nicefrac{1}{2}\\
     R & 0
   \end{pmatrix}
   \begin{pmatrix}y\\t\end{pmatrix}
  \preceq_{Q}
  \begin{pmatrix}
      \phantom-\nicefrac{1}{2}
    \\\       -\nicefrac{1}{2}
    \\0\end{pmatrix}.
\end{align*}
Concatenating this with the original constraints gives a QS function with
parameters
\begin{equation*}
       A = \pmat{0&\nicefrac{1}{2}\\0&\nicefrac12\\R&0\\A_0&0},
\quad  b = \pmat{{\phantom-\nicefrac12}\\-\nicefrac12\\0\\b_0},
\quad  d = \pmat{d_0\\-\nicefrac12},
\quad  B = \pmat{B_0\\0},
\quad  \Kscr = \Qcone^{n+2}\times \Kscr_0.
\end{equation*}

\end{document}